\documentclass[11pt]{amsart}

\usepackage{amsthm}
\usepackage{amsmath}
\usepackage{mathdots}
\usepackage{graphicx,color}
\usepackage{subfigure}

\newtheorem{thm0}{Theorem}[section]

\newtheorem{prop}[thm0]{Proposition}
\newtheorem{rem}[thm0]{Remark}
\newtheorem{cor}[thm0]{Corollary}
\newtheorem{defin}[thm0]{Definition}

\makeatletter
\g@addto@macro\@verbatim\small
\makeatother

\begin{document}

\title{Logarithmic bundles of multi-degree arrangements in $ \mathbf{P}^n $}
\author{Elena Angelini}
\address{Dipartimento di Matematica e Informatica, Universit\`a degli Studi di Ferrara - Via Machiavelli 30, 44121 Ferrara}
\email{elena.angelini@unife.it}
\maketitle

%%%%%%%%%%%%%%%%%%%%%%%%%%%%%%%%%%%%%%%%%%%%%%%%%%%%%%%%%%%%%%%%%%%%%%%%%%%%%%%%%%%%%%%%%%%%%%%%%%%%%%%%%%%%%%%%%%%%%%%%%%%%%%%%%

\begin{abstract} Let $ \mathcal{D} = \{D_{1}, \ldots, D_{\ell}\} $ be a multi-degree arrangement with normal crossings on the complex projective space $ \mathbf{P}^{n} $, with degrees $ d_{1}, \ldots, d_{\ell} $; let $ \Omega_{\mathbf{P}^{n}}^{1}(\log \mathcal{D}) $ be the logarithmic bundle attached to it. First we prove a Torelli type theorem when $ \mathcal{D} $ has a sufficiently large number of components by recovering them as unstable smooth irreducible degree-$ d_{i} $ hypersurfaces of $ \Omega_{\mathbf{P}^{n}}^{1}(\log \mathcal{D}) $. Then, when $ n = 2 $, by describing the moduli spaces containing $ \Omega_{\mathbf{P}^{2}}^{1}(\log \mathcal{D}) $, we show that arrangements of a line and a conic, or of two lines and a conic, are not Torelli. Moreover we prove that the logarithmic bundle of three lines and a conic is related with the one of a cubic. Finally we analyze the conic-case. \smallskip



\end{abstract}

\section{Introduction}

In the complex projective space $ \mathbf{P}^{n} $, let $ \mathcal{D} $ be a union of $ \ell $ distinct smooth irreducible hypersurfaces with degrees $ d_{1}, \ldots, d_{\ell} $, i.e. a \emph{multi-degree arrangement}.  We can map $ \mathcal{D} $ to $ \Omega^{1}_{\mathbf{P}^{n}}(\log \mathcal{D}) $, the \emph{sheaf of differential $ 1 $-forms with logarithmic poles on $ \mathcal{D} $}. This sheaf was originally introduced by Deligne in~\cite{De} for arrangements with normal crossings. In this case, for all $ x \in \mathbf{P}^{n} $, the space of sections of $ \Omega^{1}_{\mathbf{P}^{n}}(\log \mathcal{D}) $ near $ x $ is defined as $ <d\log z_{1}, \ldots, d\log z_{k}, dz_{k+1}, \ldots, dz_{n}>_{\mathcal{O}_{\mathbf{P}^{n},x}} $, where $ z_{1}, \ldots, z_{n} $ are local coordinates such that $ \mathcal{D} = \{z_{1} \cdot \ldots \cdot z_{k} = 0\} $. In particular, $ \Omega^{1}_{\mathbf{P}^{n}}(\log \mathcal{D}) $ is a locally free sheaf over $ \mathbf{P}^{n} $ and it is called \emph{logarithmic bundle}.\\
A natural, interesting question is whether $ \Omega^{1}_{\mathbf{P}^{n}}(\log \mathcal{D}) $ contains information enough to recover $ \mathcal{D} $, which is the so-called \emph{Torelli problem} for logarithmic bundles. In particular, if the isomorphism class of $ \Omega^{1}_{\mathbf{P}^{n}}(\log \mathcal{D}) $ determines $ \mathcal{D} $, then $ \mathcal{D} $ is called a \emph{Torelli arrangement}. \newline
\indent In the mathematical literature, the first situation that has been analyzed is the case of hyperplanes. In \cite{Do-Ka} Dolgachev, Kapranov proved that, if $ \ell \leq n+2 $, then two different arrangements always give the same logarithmic bundle and in \cite{V} Vall\`es showed that, if $ \ell \geq n+3 $, then we can reconstruct the hyperplanes from the logarithmic bundle (as  its \emph{unstable hyperplanes}, see Definition \ref{D:hypunst}) unless they osculate a rational normal curve $ \mathcal{C}_{n} $ of degree $ n $ in $ \mathbf{P}^n $, in which case the logarithmic bundle is isomorphic to $ E_{\ell - 2}(\mathcal{C}_{n}^{\vee}) $, the \emph{Schwarzenberger bundle} (\cite{Sch2}, \cite{Sch1}) \emph{of degree $ \ell-2 $ associated to $ \mathcal{C}_{n}^{\vee} $}. Recently, Dolgachev (\cite{Do}) and Faenzi, Matei, Vall\`es (\cite{FMV}) solved this problem in the case of hyperplanes that do not necessarily satisfy the normal crossings property. \\
Concerning the higher degree case, Ueda and Yoshinaga (\cite{UY1}, \cite{UY2}) studied the case $ \ell = 1 $, characterizing generically the Torelli arrangements as the ones with $ d_{1} \geq 3 $. In \cite{A1} we analyzed hypersurfaces of the same degree $ d $ and, by means of the \emph{unstable hypersurfaces} of $ \Omega^{1}_{\mathbf{P}^n}(\log \mathcal{D}) $ (see Definition \ref{D:hypunst}), we proved a Torelli type theorem when $ \ell \geq {{n+d} \choose {d}}+3 $. Pairs of quadrics are also investigated in \cite{A1}. \\
Very recently Ballico, Huh, Malaspina (\cite{BHM}) and Dimca, Sernesi (\cite{DS}), generalizing the techniques, respectively, of \cite{A1} and \cite{UY2}, answered some Torelli type questions, respectively, in the case of logarithmic bundles over quadrics or products of projective spaces and for plane curves with nodes and cusps. \\
\indent In this paper, after recalling some preliminary tools ($\S$. $2, \, 3$), we consider multi-degree arrangements with normal crossings on $ \mathbf{P}^n $ ($\S$. $ 4 $), on $ \mathbf{P}^{2} $ ($\S$. $ 5, \, 6, \, 7 $) and conic-arrangements with normal crossings on $ \mathbf{P}^{2} $ ($\S$. $ 8 $). In Theorem \ref{multi-deg thm}, by generalizing the arguments used in \cite{A1} for hypersurfaces of the same degree and by applying a \emph{reduction technique}, we prove that if the number $ \ell_{i} $ of hypersurfaces of degree $ d_{i} $ in $ \mathcal{D} $ satisfies $ \ell_{i} \geq {{n+d_{i}} \choose {d_{i}}}+3 $, then we can generically recover the components of $ \mathcal{D} $. In $\S$. $ 5, \, 6, \, 7 $ we focus on some line-conic cases on $ \mathbf{P}^{2} $ and we prove that they are not of Torelli type (Corollaries \ref{C:line-conic}, \ref{C:2lines-conic}). In particular, in Theorem \ref{T:3lines-conic} we show a link between arrangements of three lines and a conic and arrangements with a cubic in the projective plane. Finally, $\S$. $ 8 $ is devoted to conics. The cases $ \ell \in \{1,2\} $ were studied in \cite{A1}; here we prove that for $ \ell \geq 4 $ a Torelli type result holds (Theorem \ref{T:molteconiche}). $ \ell = 3 $ is still a bit mysterious.\\

\noindent {\bf Acknowledgements} I am very grateful to Giorgio Ottaviani and Daniele Faenzi for introducing me to this interesting subject and for their help during the preparation of this work. I also thank Jean Vall\`es for several helpful comments.

\section{Preliminary definitions and notations}

Let $ \mathbf{P}^{n} $ be the $ n $-dimensional complex projective space with $ n \geq 2 $ and let $ \mathcal{D} = \{D_{1}, \ldots, D_{\ell}\} $ be an \emph{arrangement} on $ \mathbf{P}^{n} $, i.e. a family of smooth, irreducible, distinct hypersurfaces of $ \mathbf{P}^{n} $. Let us assume that $ \mathcal{D} $ has \emph{normal crossings}, that is $ \mathcal{D} $ is locally isomorphic (in the sense of holomorphic local coordinates changes) to a union of coordinate hyperplanes of $ \mathbf{C}^n $.\\ For all $ i \in \{1, \ldots, \ell \} $, $ D_{i} = \{f_{i} = 0\}$ with $ f_{i} \in \mathbf{C}[x_{0}, \ldots, x_{n}]_{d_{i}} $ for certain $ d_{i} $; thus $ \mathcal{D} = \{f = 0\} $, where $ f = f_{1}\cdot \ldots \cdot f_{\ell} $ has degree $ d = d_{1}+ \ldots + d_{\ell} $. In particular, if all $ d_{i} $'s are equal to $ 1 $ we speak of a \emph{hyperplane arrangement}, if they are equal to $ 2 $ we deal with an \emph{arrangement of quadrics} and so on. If different $ d_{i} $'s appear in $ \mathcal{D} $, then we call $ \mathcal{D} $ a \emph{multi-degree arrangement}. \\
\indent In order to introduce the notion of \emph{sheaf of logarithmic forms} on $ \mathcal{D} $ we refer to Deligne (\cite{De0},~\cite{De}). Let $ U $ be the complement of $ \mathcal{D} $ in $ \mathbf{P}^{n} $ and let $ j $ be the embedding of $ U $ in $ \mathbf{P}^{n} $. We denote by $ \Omega_{U}^1$ the sheaf of holomorphic differential $ 1 $-forms on $ U $ and by $ j_{\ast}\Omega_{U}^1 $ its direct image sheaf on $ \mathbf{P}^{n} $. Since $ \mathcal{D} $ has normal crossings, then for all $ x \in \mathbf{P}^{n} $ there exists a Euclidean neighbourhood $ I_{x} \subset \mathbf{P}^{n} $ such that $ I_{x} \cap \mathcal{D} = \{z_{1} \cdots z_{k} = 0 \} $, where $ \{z_{1}, \ldots, z_{k}\} $ is a part of a system of local coordinates. We have the following:
\begin{defin}\label{d:1}
The \emph{sheaf of differential $ 1 $-forms on $ \mathbf{P}^{n} $ with logarithmic poles on $ \mathcal{D} $} is the subsheaf $ \Omega_{\mathbf{P}^{n}}^{1}(\log \mathcal{D}) $ of $ j_{\ast}\Omega_{U}^1 $, such that, for all $ x \in \mathbf{P}^{n} $,
$$ \Gamma(I_{x},\Omega_{\mathbf{P}^{n}}^{1}(\log \mathcal{D})) = \{s \in \Gamma(I_{x},j_{\ast}\Omega_{U}^1) \,|\, s = \displaystyle{\sum_{i=1}^k u_{i} d\log z_{i}} \, +  \displaystyle{\sum_{i=k+1}^n v_{i} dz_{i}} \} $$
where $ u_{i}, v_{i} $ are locally holomorphic functions and $ d \log z_{i} = \displaystyle{{dz_{i}} \over {z_{i}}} $.
\end{defin}
\noindent Another way to describe these sheaves, which is useful for more general divisors and is equivalent to the previous one in the normal crossings case, is the following, (\cite{Sa},~\cite{Schenck}): 
\begin{defin}\label{d:Gaussmap}
The \emph{sheaf of diff. $ 1 $-forms on $ \mathbf{P}^{n} $ with log. poles on} $ \mathcal{D} $ is
$$ \Omega_{\mathbf{P}^n}^{1}(\log \mathcal{D}) = \mathcal{T}(\log\mathcal{D})^{\vee}(-1), $$
where $ \mathcal{T}(\log\mathcal{D}) $ is the kernel of the \emph{Gauss map} $ \mathcal{O}_{\mathbf{P}^{n}}^{n+1} \xrightarrow{(\partial_{0}{\it f}, \ldots, \partial_{{\it n}}{\it f})} \mathcal{O}_{\mathbf{P}^{n}}(d-1). $ 
\end{defin} 
\noindent Since $ \mathcal{D} $ has normal crossings, $ \Omega_{\mathbf{P}^{n}}^{1}(\log \mathcal{D}) $ is a locally free sheaf of rank $ n$,~\cite{De}. It is called the \emph{logarithmic bundle attached to $ \mathcal{D} $}. \\
Definition \ref{d:1} can be used, more generally, to introduce the logarithmic bundle of an arrangement with normal crossings $ \mathcal{D} $ on a smooth algebraic variety $ X $ (see also \cite{A1}). \\
\indent Our investigations are mainly based on the following:
\begin{thm0}\label{T:base}
$ \Omega_{\mathbf{P}^n}^{1}(\log \mathcal{D}) $ admits the short exact sequences \\
\begin{equation}\label{eq:residueexactsequence}
0 \longrightarrow \Omega_{\mathbf{P}^{n}}^1 \longrightarrow \Omega_{\mathbf{P}^{n}}^{1}(\log \mathcal{D}) \buildrel \rm {\it res} \over \longrightarrow \, \displaystyle{\bigoplus_{i=1}^{\ell} \mathcal{O}_{D_{i}}} \longrightarrow 0,
\end{equation}
where res denotes the \emph{Poincar\'e residue morphism} $($\cite{Do-Ka}$)$ and
\begin{equation}\label{eq:Anconas}
0 \longrightarrow \Omega_{\mathbf{P}^n}^{1}(\log \mathcal{D})^{\vee} \longrightarrow \mathcal{O}_{\mathbf{P}^n}(1)^{n+1} \oplus \mathcal{O}_{\mathbf{P}^n}^{\ell-1} \buildrel \rm {\it N} \over \longrightarrow \displaystyle{\bigoplus_{i=1}^{\ell} \mathcal{O}_{\mathbf{P}^n}(d_{i})} \longrightarrow 0, 
\end{equation}
where $ N $ is a $ \ell \times (n + \ell) $ matrix depending on the $ f_{j} $'s and their partial derivatives $($\cite{A1}$)$.
\end{thm0}
Our aim is to study the injectivity of the correspondence
\begin{equation}\label{eq:Torellimap}
\mathcal{D} \longmapsto \Omega_{\mathbf{P}^{n}}^{1}(\log \mathcal{D}) 
\end{equation}
where $ \mathcal{D} $ is a multi-degree arrangement with normal crossings with fixed degrees $ d_{1}, \ldots, d_{\ell} $, that is the \emph{Torelli problem for logarithmic bundles}. In the case of $ 1:1 $ correspondence we call $ \mathcal{D} $ an \emph{arrangement of Torelli type} or, simply, a \emph{Torelli arrangement}. \\
\indent In the next section we recall the main results concerning this problem in the case of hyperplanes (\cite{Do-Ka}, \cite{V}, \cite{AO}), of one smooth hypersurface (\cite{UY1}, \cite{UY2}, \cite{A1}), of \emph{many} smooth hypersurfaces of degree $ d \geq 2 $ and of two smooth quadrics (\cite{A1}). In some of them, the components of $ \mathcal{D} $ are recovered by looking at the set of \emph{unstable objects} of $ \Omega_{\mathbf{P}^{n}}^{1}(\log \mathcal{D}) $ of a given degree; to that end we make the following:  
\begin{defin}\label{D:hypunst}
Let $ D \subset \mathbf{P}^{n} $ be a hypersurface. We call $ D $ \emph{unstable} for $ \Omega_{\mathbf{P}^{n}}^{1}(\log \mathcal{D}) $ if the following condition holds: 
\begin{equation}\label{eq:hypunst}
H^{0}(D, {\Omega_{\mathbf{P}^n}^{1}(\log \mathcal{D})}^{\vee}_{|_{D}}) \not= \{0\}.
\end{equation}
\end{defin} 
\begin{rem}\label{r53}
Let us suppose that $ \mathcal{D} $ has $ \ell = \ell_{1} + \ldots + \ell_{m} $ components such that $ \ell_{i} $ have degree $ d_{i} $, $ i \in \{1, \ldots, m\} $. We are interested in Definition \ref{D:hypunst} when $ h^{0}(\mathbf{P}^{n},{\Omega_{\mathbf{P}^n}^{1}(\log \mathcal{D})}^{\vee}) = \{0\} $, that is, by using the same arguments of Remark 5.3 of \cite{A1}, when
\begin{equation}\label{eq:inequality}
\displaystyle{\sum_{i=1}^{m} ({\ell_{i}} \cdot d_{i}) > n +1}.
\end{equation}
\end{rem}
\begin{rem}\label{R:leftinclusion}
In Lemma 5.4 of \cite{A1}, by means of (\ref{eq:residueexactsequence}) we prove that each component $ D_{i} $ of $ \mathcal{D} $ is an unstable hypersurface of degree $ d_{i} $ for $ \Omega_{\mathbf{P}^{n}}^{1}(\log \mathcal{D}) $. \\
As in Definition \ref{D:hypunst}, we can introduce the notion of unstable hypersurface for $ \Omega_{X}^{1}(\log \mathcal{D}) $ when $ X $ is a smooth algebraic variety and $ \mathcal{D} $ is an arrangement with normal crossings on it. In a similar way we can prove that each element of $ \mathcal{D} $ is unstable for $ \Omega_{X}^{1}(\log \mathcal{D}) $.
\end{rem}

\section{Some known Torelli type results}

Let $ \mathcal{H} = \{H_{1}, \ldots, H_{\ell}\} $ be a hyperplane arrangement with normal crossings on $ \mathbf{P}^n $. If $ \ell \leq n+2 $, then $ \mathcal{H} $ isn't of Torelli type (\cite{Do-Ka}); otherwise we have the following result (\cite{V}, Theorem 3.1):
\begin{thm0}\label{T:hyperplanes}
If $ \ell \geq n+3 $ then $\mathcal{H} $ is the set of unstable hyperplanes of $ \Omega_{\mathbf{P}^n}^{1}(\log \mathcal{H}) $, unless $ H_{1}, \ldots, H_{\ell} $ osculate a rational normal curve $ \mathcal{C}_{n}\subset \mathbf{P}^{n} $ of degree $ n $, in which case all the hyperplanes lying on $ \mathcal{C}_{n}^{\vee} \subset (\mathbf{P}^{n})^{\vee} $ are unstable and $ \Omega_{\mathbf{P}^n}^{1}(\log \mathcal{H}) \cong E_{\ell - 2} (\mathcal{C}_{n}^{\vee}) $, the \emph{Schwarzenberger bundle of degree} $ \ell - 2 $ \emph{associated to} $ \mathcal{C}_{n}^{\vee} $.
\end{thm0}
\indent If $ \mathcal{D} = \{D_{1}\} $, where $ D_{1} \subset \mathbf{P}^{n} $ is a general hypersurface of degree $ d_{1} $, then $ \mathcal{D} $ is of Torelli type if and only if $ d_{1} \geq 3 $ (\cite{UY2}, Theorem 1; \cite{A1}, Proposition 6.1). \\
\indent Now, let $ \mathcal{D} = \{D_{1}, \ldots, D_{\ell}\} $  be an arrangement with normal crossings on $ \mathbf{P}^{n} $, with $ \ell \geq 2 $ and $ d_{i} = d \geq 2 $ for all $ i \in \{1, \ldots, \ell\} $. By associating to $ \mathcal{D} $ a hyperplane arrangement $ \mathcal{H} $ in $ \mathbf{P}^{ {{n+d} \choose {d}}-1 } $ through the $ d $\emph{-uple Veronese embedding} and by applying Theorem \ref{T:hyperplanes}, we get the following result (\cite{A1}, Theorem 5.5):
\begin{thm0}\label{T:samedegree} If $ \ell \geq {{n+d} \choose {d}} + 3 $ and $ \mathcal{H} $ is a hyperplane arrangement with normal crossings whose components don't osculate a rational normal curve of degree $ {{n+d} \choose {d}}-1 $ in $ \mathbf{P}^{ {{n+d} \choose {d}}-1 } $, then $ \mathcal{D} $ is the set of smooth, irreducible, degree-$ d $ hypersurfaces of $ \mathbf{P}^{n} $ unstable for $ \Omega_{\mathbf{P}^n}^{1}(\log \mathcal{D}) $.
\end{thm0}
\indent In (\cite{A1}, Theorem 7.5) we prove also that if $ \ell = d = 2 $ then $ \mathcal{D} $ is not a Torelli arrangement. Indeed, by using the \emph{simultaneous diagonalization} of the matrices of the smooth quadrics and a \emph{duality} argument, we get that two such arrangements have isomorphic logarithmic bundles if and only if they have the same tangent hyperplanes. \\
\indent In the next sections we present some recent results concerning multi-degree arrangements ($\S$. 4, 5, 6, 7) and an almost complete description of the conic-case ($\S$. 8).

\section{Many multi-degree hypersurfaces}

Let $ \mathcal{D} = \{D_{1}^{d_{1}}, \ldots, D^{d_{1}}_{\ell_{1}}, D_{1}^{d_{2}}, \ldots, D^{d_{2}}_{\ell_{2}}, \ldots \ldots, D_{1}^{d_{m}}, \ldots, D^{d_{m}}_{\ell_{m}}\} $ be a multi-degree arrangement with normal crossings in $ \mathbf{P}^{n} $ such that the components $ D_{1}^{d_{i}}, \ldots, D^{d_{i}}_{\ell_{i}} $ have degree $ d_{i} $, with $ i \in \{1, \ldots, m\} $ and $ d_{m} > d_{m-1} > \cdots > d_{1} $; let us denote by $ \Omega_{\mathbf{P}^n}^{1}(\log \mathcal{D}) $ the corresponding logarithmic bundle. \\
When the number of components in $ \mathcal{D} $ is \emph{sufficiently large}, the Torelli problem can be solved by generalizing the method used in \cite{A1} and by applying a \emph{reduction technique} inspired by the one adopted in \cite{V}. So, let $ \mathcal{H}_{d_{i}} $ be the arrangement with $ \ell_{i} $ hyperplanes on $ \mathbf{P}^{N_{i}} $, with $ i \in \{1, \ldots, m\} $ and $ N_{i} = {{n+d_{i}} \choose {d_{i}}} - 1 $, associated to $ \{D_{1}^{d_{i}}, \ldots, D^{d_{i}}_{\ell_{i}}\} $ by means of the $ d_{i} $-uple Veronese embedding, i.e. $ \nu_{d_{i}} : \mathbf{P}^n \longrightarrow \mathbf{P}^{N_{i}} $ and $ \nu_{d_{i}}([x_{0}, \ldots, x_{n}]) = [\ldots \, x^{I} \ldots] $, where $ x^{I} $ ranges over all monomials of degree $ d_{i} $ in $ x_{0},\ldots, x_{n} $. Let us assume that each $ \mathcal{H}_{d_{i}} $ has normal crossings on $ \mathbf{P}^{N_{i}} $ and let $ \Omega_{\mathbf{P}^{N_{i}}}^{1}(\log \mathcal{H}_{d_{i}}) $ be the associated logarithmic bundle. With the previous notation, let us consider the diagonal embedding: 
$$ \nu : \mathbf{P}^n \longrightarrow \mathbf{P} = \displaystyle{\prod_{i=1}^{m} \mathbf{P}^{N_{i}}} $$
$$ \nu([x_{0}, \ldots, x_{n}]) = [\nu_{d_{1}}([x_{0}, \ldots, x_{n}]), \ldots, \nu_{d_{m}}([x_{0}, \ldots, x_{n}])]. $$
Let $ p_{i} : \mathbf{P} \longrightarrow {\mathbf{P}^{N_{i}}} $ be the $ i $-th projection and let $ h_{i} = c_{1}(p_{i}^{*}(\mathcal{O}_{\mathbf{P}^{N_{i}}}(1))) $. By means of $ \nu $, we can associate to the multi-degree arrangement $ \mathcal{D} $ an arrangement $ \mathcal{A} = \mathcal{A}_{1} \cup \ldots \cup \mathcal{A}_{m} $ 
on $ \mathbf{P} $ such that $ \mathcal{A}_{i} $ is an irreducible divisor of class $ h_{i} $ which is the pull-back via $ p_{i} $ of $ \mathcal{H}_{d_{i}} $. \\
Let us assume that $ \mathcal{A} $ has normal crossings and let $ \Omega_{\mathbf{P}}^{1}(\log \mathcal{A}) $ be its logarithmic bundle (see also \cite{BHM} for some results concerning logarithmic bundles over products of projective spaces). 
\begin{rem} 
The following property holds:
\begin{equation}\label{eq:scomposizione}
\Omega_{\mathbf{P}}^{1}(\log \mathcal{A}) \cong \displaystyle{\bigoplus_{i=1}^{m} p_{i}^{\ast}(\Omega_{\mathbf{P}^{N_{i}}}^{1}(\log \mathcal{H}_{d_{i}}))}. 
\end{equation}
Moreover, if $\ell_{i} \geq N_{i} + 2 $, $\mathcal{H}_{d_{i}} $ having normal crossings, $ \Omega_{\mathbf{P}^{N_{i}}}^{1}(\log \mathcal{H}_{d_{i}}) $ is a \emph{Steiner bundle} over $ \mathbf{P}^{N_{i}} $, \cite{Do-Ka}. So, because of (\ref{eq:scomposizione}), $\Omega_{\mathbf{P}}^{1}(\log \mathcal{A})$ admits the short exact sequence
\begin{equation}\label{eq:Steiner1}
0 \longrightarrow \displaystyle{\bigoplus_{i=1}^{m}}\mathcal{O}_{\mathbf{P}}(-h_{i})^{\ell_{i}-N_{i}-1} \longrightarrow \displaystyle{\bigoplus_{i=1}^{m}} \mathcal{O}_{\mathbf{P}}^{\ell_{i}-1} \longrightarrow \Omega_{\mathbf{P}}^{1}(\log \mathcal{A}) \longrightarrow 0.
\end{equation}
\end{rem}
Now we can state and prove the main result concerning the Torelli problem for multi-degree arrangements with \emph{many} components.

\begin{thm0}\label{multi-deg thm}
Let $ \mathcal{D} $ be a multi-degree arrangement with normal crossings on $ \mathbf{P}^n $ and let $ \mathcal{H}_{d_{1}}, \ldots, \mathcal{H}_{d_{m}} $, $ \mathcal{A} $ be the corresponding arrangements, respectively, on $ \mathbf{P}^{N_{1}}, \ldots, \mathbf{P}^{N_{m}} $ and $ \mathbf{P} $, in the sense of Veronese maps. \\Assume that, for all $ i \in \{1, \ldots, m\} $$:$\\
\noindent $1.$ $ \ell_{i} \geq N_{i}+ 4 $ \\
\noindent $2.$ $ \mathcal{A} $ has normal crossings on $ \mathbf{P} $ \\
\noindent $3.$ $ \mathcal{H}_{d_{i}} $ has normal crossings on $ \mathbf{P}^{N_{i}} $ and its hyperplanes don't osculate a rational normal curve of degree $ N_{i} $ in $ \mathbf{P}^{N_{i}} $.\\
Then $ \mathcal{D} = \{D \subset \mathbf{P}^{n} \, smooth \, irred. \, hypers. \, of \, degree\, d_{i}, \, \exists\, i \, | \, D \, satisfies \, (\ref{eq:hypunst}) \}. $
\end{thm0}

\begin{proof}
We perform a double inclusion argument between the two sets in the last line of the statement of Theorem \ref{multi-deg thm}. We observe that the inclusion $ \subset $ follows from Remark \ref{R:leftinclusion}. \\
\indent Thus, let us assume that $ D \subset \mathbf{P}^n $ is a smooth irreducible hypersurface of degree $ d_{i} $ which is unstable for $ \Omega_{\mathbf{P}^n}^{1}(\log \mathcal{D}) $, we want to prove that $ D \in \mathcal{D} $. \\
\indent First let us suppose that the degree of $ D $ is the highest one, i.e. $ d_{m} $. \\Our aim is to show that, denoting by $ H_{d_{m}} \subset \mathbf{P}^{N_{m}} $ the hyperplane associated to $ D $ by means of $ \nu_{d_{m}} $, then $ H = p^{*}_{m}(H_{d_{m}}) \subset \mathbf{P} $ satisfies
\begin{equation}
H^{0}(H,\Omega_{\mathbf{P}}^{1}(\log \mathcal{A})^{\vee}_{|_{H}}) \not=\{0\}.
\end{equation}
Indeed, if this is the case, $ H_{d_{m}} $ is an unstable hyperplane for $ \Omega_{\mathbf{P}^{N_{m}}}^{1}(\log \mathcal{H}_{d_{m}}) $ and so hypothesis $1.$ and $ 3. $ allow us to apply Theorem \ref{T:hyperplanes}, which implies that $ H_{d_{m}} \in \mathcal{H}_{d_{m}} $. In particular, we get that $ D = D_{j}^{d_{m}} \in \mathcal{D} $ for certain $ j \in \{1, \ldots, \ell_{m}\} $. \\
Let us denote by $ V $ the image of the map $ \nu $; since $ V $ is a non singular subvariety of $ \mathbf{P} $ which intersects transversally $ \mathcal{A} $, from Proposition $ 2.11 $ of~\cite{Do} we get the following exact sequence
\begin{equation}\label{eq:prop211Dolg}
0 \longrightarrow  \mathcal{N}_{V, \, \mathbf{P}}^{\vee} \longrightarrow \Omega_{\mathbf{P}}^{1}(\log \mathcal{A})_{|_{V} } \longrightarrow \Omega_{V}^{1}(\log \mathcal{A} \cap V) \longrightarrow 0 
\end{equation}
where $ \mathcal{N}_{V, \, \mathbf{P}}^{\vee} $ denotes the conormal sheaf of $ V $ in $ \mathbf{P} $. \\
We remark that $ V \cong \mathbf{P}^n $ and $ \mathcal{D} = \mathcal{A} \cap V $, so if we restrict (\ref{eq:prop211Dolg}) to $ D $, we apply $ \mathcal{H}om(\cdot, \, \mathcal{O}_{D}) $ and then we pass to cohomology we get
$$ 0 \longrightarrow H^{0}(D, {\Omega_{\mathbf{P}^n}^{1}(\log \mathcal{D})}^{\vee}_{|_{D}}) \longrightarrow H^{0}(D, {\Omega_{\mathbf{P}}^{1}(\log \mathcal{A})}^{\vee}_{|_{D}}). $$
Since $ D $ is unstable for $ \Omega_{\mathbf{P}^n}^{1}(\log \mathcal{D}) $, necessarily it has to be 
\begin{equation}\label{eq:keyfact1}
H^{0}(D, {\Omega_{\mathbf{P}}^{1}(\log \mathcal{A})}^{\vee}_{|_{D}}) \not= \{0\}.
\end{equation}
Now, let us tensor the ideal sheaf sequence of $ V $ in $ \mathbf{P} $ with $ {\Omega_{\mathbf{P}}^{1}(\log \mathcal{A})}^{\vee}_{|_{H}} $; we have the exact sequence
$$ 0 \longrightarrow \mathcal{I}_{V \cap H, \, H}\, \otimes \, {\Omega_{\mathbf{P}}^{1}(\log \mathcal{A})}^{\vee}_{|_{H}} \longrightarrow {\Omega_{\mathbf{P}}^{1}(\log \mathcal{A})}^{\vee}_{|_{H}} \longrightarrow {\Omega_{\mathbf{P}}^{1}(\log \mathcal{A})}^{\vee}_{|_{D}} \longrightarrow 0. $$
Passing to cohomology we get
$$ 0 \longrightarrow H^{0}(H, \mathcal{I}_{V \cap H, \, H} \otimes {\Omega_{\mathbf{P}}^{1}(\log \mathcal{A})}^{\vee}_{|_{H}}) \longrightarrow H^{0}(H, {\Omega_{\mathbf{P}}^{1}(\log \mathcal{A})}^{\vee}_{|_{H}}) \longrightarrow $$
$$ \,\,\,\,\,\, \longrightarrow H^{0}(D, {\Omega_{\mathbf{P}}^{1}(\log \mathcal{A})}^{\vee}_{|_{D}}) \longrightarrow H^{1}(H, \mathcal{I}_{V \cap H, \, H} \otimes \, {\Omega_{\mathbf{P}}^{1}(\log \mathcal{A})}^{\vee}_{|_{H}}). \quad\,\,\,\,  $$
We remark that to conclude the proof it suffices to show that 
\begin{equation}\label{eq:keyfact2}
H^{1}(H, \mathcal{I}_{V \cap H, \, H} \otimes {\Omega_{\mathbf{P}}^{1}(\log \mathcal{A})}^{\vee}_{|_{H}}) = \{0\}.
\end{equation}
Since hypothesis $1.$ and $ 3. $ hold, we are allowed to use (\ref{eq:Steiner1}), which, by applying $ \mathcal{H}om(\cdot, \, {\mathbf{P}}) $ turns out to be
$$ 0 \longrightarrow {\Omega_{\mathbf{P}}^{1}(\log \mathcal{A}})^{\vee} \longrightarrow \displaystyle{\bigoplus_{i=1}^{m}} \mathcal{O}_{{\mathbf{P}}}^{\ell_{i}-1} \longrightarrow \displaystyle{\bigoplus_{i=1}^{m}}\mathcal{O}_{\mathbf{P}}(h_{i})^{\ell_{i}-N_{i}-1} \longrightarrow 0. $$
If we tensor with $ \mathcal{I}_{V,\,\mathbf{P} \,|_{H}} $ and then we pass to cohomology, the previous sequence becomes 
\begin{equation}\label{eq:coomolSteiner1}
\cdots \longrightarrow \displaystyle{\bigoplus_{i=1}^{m}} H^{0}(H, \mathcal{I}_{V \cap H, \, H} \otimes \, {\mathcal{O}_{\mathbf{P}}(h_{i})_{|_{H}}^{\ell_{i}-N_{i}-1}}) \longrightarrow 
\end{equation}
$$ \,\,\, \longrightarrow H^{1}(H, \mathcal{I}_{V \cap H, \, H} \otimes \, {{\Omega_{\mathbf{P}}^{1}(\log \mathcal{A}})^{\vee}}_{|_{H}}) \longrightarrow \displaystyle{\bigoplus_{i=1}^{m}} H^{1}(H, \mathcal{I}_{V \cap H, \, H} \otimes \, {\mathcal{O}_{{\mathbf{P}}}^{\ell_{i}-1}}_{|_{H}}). $$ 
In order to prove (\ref{eq:keyfact2}) it suffices to show that
\begin{equation}\label{eq:0}
H^{1-k}(H, \mathcal{I}_{V \cap H, \, H}(kh_{i})) = \{0\}
\end{equation}
for $ k = \{0,1\} $ and for all $ i \in \{1, \ldots, m\} $. $ V \cap H $ being connected, from the induced cohomology sequence of the ideal sheaf sequence of $ V $ in $ \mathbf{P} $, restricted to $ H $, we immediately get (\ref{eq:0}) for $ k=0 $.\\
So, let us consider the exact commutative diagram:
$$ 
\begin{matrix}
{} & {} & 0 & {} & 0 & {} & 0 & {} & {} \cr
{} & {} & \downarrow & {} & \downarrow & {} & \downarrow & {} & {} \cr
0 & \rightarrow & \mathcal{I}_{V, \, \mathbf{P}}(h_{i}-h_{m}) & \rightarrow & \mathcal{O}_{\mathbf{P}}(h_{i}-h_{m}) & \rightarrow & \mathcal{O}_{\mathbf{P}^{n}}(d_{i}-d_{m}) & \rightarrow & 0 \cr
{} & {} & \downarrow & {} & \downarrow & {} & \downarrow & {} & {} \cr
0 & \rightarrow & \mathcal{I}_{V, \, \mathbf{P}}(h_{i}) & \rightarrow & \mathcal{O}_{\mathbf{P}}(h_{i}) & \rightarrow & \mathcal{O}_{\mathbf{P}^{n}}(d_{i}) & \rightarrow & 0 \cr
{} & {} & \downarrow & {} & \downarrow & {} & \downarrow & {} & {} \cr
0 & \rightarrow & \mathcal{I}_{V \cap H, \, H}(h_{i}) & \rightarrow & \mathcal{O}_{H}(h_{i}) & \rightarrow & \mathcal{O}_{V \cap H}(d_{i}) & \rightarrow & 0 \cr
{} & {} & \downarrow & {} & \downarrow & {} & \downarrow & {} & {} \cr
{} & {} & 0 & {} & 0 & {} & 0 & {} & {} \cr
\end{matrix}
$$
Since $ H^{0}(\mathbf{P},\mathcal{O}_{\mathbf{P}}(h_{i})) \rightarrow H^{0}(\mathbf{P}^{n},\mathcal{O}_{\mathbf{P}^{n}}(d_{i})) $ is an isomorphism, we always get
\begin{equation}\label{eq:1}
H^{0}(\mathbf{P}, \mathcal{I}_{V, \, \mathbf{P}}(h_{i})) = H^{1}(\mathbf{P},\mathcal{I}_{V, \, \mathbf{P}}(h_{i})) = \{0\}. 
\end{equation}
Moreover, looking at the first row of the diagram, we obtain, for all $ i $, 
\begin{equation}\label{eq:2}
H^{1}(\mathbf{P},\mathcal{I}_{V, \, \mathbf{P}}(h_{i}-h_{m})) = \{0\}
\end{equation}
By using (\ref{eq:1}) and (\ref{eq:2}), the first column of the diagram implies (\ref{eq:0}) for $ k=1 $, as desired. \\
\indent Now, let us suppose that $ D $ has degree $ d_{i} $ with $ i \in \{m-1, m-2, \ldots, 1 \} $. In order to prove that $ D \in \mathcal{D} $, we apply a \emph{reduction technique} to $ \Omega_{\mathbf{P}^n}^{1}(\log \mathcal{D}) $ and to the hypersurfaces of $ \mathcal{D} $ of highest degree $ d_{m} $. Let's start with $ D_{\ell_{m}}^{d_{m}} $: since for this hypersurface (\ref{eq:hypunst}) holds, there exists a non-zero surjective homomorphism
$$ \Omega_{\mathbf{P}^n}^{1}(\log \mathcal{D})_{|_{D_{\ell_{m}}^{d_{m}}}} \longrightarrow \, \mathcal{O}_{D_{\ell_{m}}^{d_{m}}}, $$
which induces a surjective composed homomorphism $ g_{\ell_{m}} $
$$ \Omega_{\mathbf{P}^n}^{1}(\log \mathcal{D}) \longrightarrow \Omega_{\mathbf{P}^n}^{1}(\log \mathcal{D})_{|_{D_{\ell{m}}^{d_{m}}}} \longrightarrow \, \mathcal{O}_{D_{\ell{m}}^{d_{m}}}. $$
Its kernel, denoted by $ K_{\ell_{m}}^{d_{m}} $, turns out to be a rank-$ n $ vector bundle over $ \mathbf{P}^{n} $. \\
If we apply the \emph{snake lemma} to the commutative diagram 
$$ 0 \longrightarrow \displaystyle{\bigoplus_{i=1}^{m} \mathcal{O}_{\mathbf{P}^n}(-d_{i})^{\ell_{i}}} \longrightarrow \mathcal{O}_{\mathbf{P}^{n}}(-1)^{n+1} \, \oplus \,  \mathcal{O}_{\mathbf{P}^{n}}^{\left( \sum_{i=1}^{m}\ell_{i} \right)-1} \longrightarrow \Omega_{\mathbf{P}^n}^{1}(\log \mathcal{D}) \longrightarrow 0  $$
$$ \quad\,\,\,\,\,\,\,\,\,\,\,\,\,\,\,\,\,\,\,\,\,\,\,\,\,\,\,\,\,\,\, \downarrow \,\,\,\,\,\,\,\,\,\,\,\,\,\,\,\,\,\,\,\,\,\,\,\,\,\,\,\,\,\,\,\,\,\,\,\,\,\,\,\,\,\,\,\,\,\,\,\,\,\,\,\,\,\,\,\,\,\,\,\,\,\,\,\,\, \downarrow \,\,\,\,\,\,\,\,\,\,\,\,\,\,\,\,\,\,\,\,\,\,\,\,\,\,\,\,\,\,\,\,\,\,\,\,\,\,\,\,\,\,\,\,\,\,\,\,\,\,\,\,\,\,\,\,\,\, \downarrow g_{\ell_{m}} \quad\quad\quad\quad   $$
$$ 0 \longrightarrow \quad \mathcal{O}_{\mathbf{P}^{n}}(-d_{m}) \,\,\,\,\,\, \longrightarrow \,\,\,\,\,\,\,\,\,\,\,\,\,\,\,\,\,\,\,\,\,\,\,\,\,\,\,\,\,\, \mathcal{O}_{\mathbf{P}^{n}} \,\,\,\,\,\,\,\,\,\,\,\,\,\,\,\,\,\,\,\,\,\,\,\,\,\,\,\,\,\, \longrightarrow \,\,\,\,\,\,\,\,\,\, \mathcal{O}_{D_{\ell_{m}}^{d_{m}}} \,\,\, \longrightarrow 0 \quad\quad\quad\quad\quad\quad\quad\quad\quad\quad\quad\quad\quad $$
we get that $ K_{\ell_{m}}^{d_{m}} $ admits the short exact sequence
$$ 0 \longrightarrow \displaystyle{\bigoplus_{i=1}^{m-1} \mathcal{O}_{\mathbf{P}^n}(-d_{i})^{\ell_{i}}} \, \oplus \,  \mathcal{O}_{\mathbf{P}^{n}}(-d_{m})^{\ell_{m}-1} \buildrel \rm \it{M}_{\ell_{m}} \over \longrightarrow \quad\quad\quad\quad\quad\quad $$
$$ \buildrel \rm \it{M}_{\ell_{m}} \over \longrightarrow \mathcal{O}_{\mathbf{P}^{n}}(-1)^{n+1} \, \oplus \, \mathcal{O}_{\mathbf{P}^{n}}^{\left( \sum_{i=1}^{m-1}\ell_{i} \right)+(\ell_{m}-1)-1} \longrightarrow K_{\ell_{m}}^{d_{m}} \longrightarrow 0 $$
where $ M_{\ell_{m}} $ is the $ \left[n+\left( \sum_{i=1}^{m-1}\ell_{i} \right)+(\ell_{m}-1) \right] \times \left[ \left( \sum_{i=1}^{m-1}\ell_{i} \right)+(\ell_{m}-1) \right] $ matrix obtained from the transpose of the matrix in (\ref{eq:Anconas}) by removing the last column and row. So we have that 
$$ K_{\ell_{m}}^{d_{m}} \cong \Omega_{\mathbf{P}^n}^{1}(\log \{D_{1}^{d_{1}}, \ldots, D^{d_{1}}_{\ell_{1}}, D_{1}^{d_{2}}, \ldots, D^{d_{2}}_{\ell_{2}}, \ldots \ldots, D_{1}^{d_{m}}, \ldots, D^{d_{m}}_{\ell_{m}-1}\}), $$
i.e. $ K_{\ell_{m}}^{d_{m}} $ is the logarithmic bundle associated to $ \mathcal{D} - \{D^{d_{m}}_{\ell_{m}}\}. $
In particular, $ D $ satisfies the condition 
\begin{equation}\label{eq:instKm}
H^{0}(D, {{K_{\ell_{m}}^{d_{m}}}^{\vee}}_{|_{D}} ) \not= \{0\}, 
\end{equation}
that is $ D $ is unstable for $ K_{\ell_{m}}^{d_{m}}. $
Indeed, if we apply $ \mathcal{H}om(\cdot, \mathcal{O}_{\mathbf{P}^{n}}) $ to the short exact sequence
$$ 0 \longrightarrow K_{\ell_{m}}^{d_{m}} \longrightarrow \Omega_{\mathbf{P}^n}^{1}(\log \mathcal{D}) \xrightarrow{\emph{g}_{\ell_{\emph{m}}}} \mathcal{O}_{D_{\ell_{m}}^{d_{m}}} \longrightarrow 0 $$
we get 
\begin{equation}\label{eq:sequence}
0 \longrightarrow \Omega_{\mathbf{P}^n}^{1}(\log \mathcal{D})^{\vee} \longrightarrow {K_{\ell_{m}}^{d_{m}}}^{\vee} \longrightarrow \mathcal{O}_{D_{\ell_{m}}^{d_{m}}}(d_{m}) \longrightarrow 0. 
\end{equation}
So, if we restrict (\ref{eq:sequence}) to $ D $ and then consider the induced cohomology sequence, we obtain an injective map
$$ H^{0}(D, \Omega_{\mathbf{P}^n}^{1}(\log \mathcal{D})^{\vee}_{|_{D}}) \longrightarrow H^{0}(D, {{K_{\ell_{m}}^{d_{m}}}^{\vee}}_{|_{D}}), $$
which implies (\ref{eq:instKm}). \\
Now, starting from $ K_{\ell_{m}}^{d_{m}} $, we iterate this technique for $ D^{d_{m}}_{\ell_{m}-1}, D^{d_{m}}_{\ell_{m}-2}, \ldots, D^{d_{m}}_{1} $ and we get a sequence of rank-$ n $ vector bundles $ K_{\ell_{m}-1}^{d_{m}}, K_{\ell_{m}-2}^{d_{m}}, \ldots, K_{1}^{d_{m}} $ over $ \mathbf{P}^{n} $ such that, for all $ s \in \{1,\ldots, \ell_{m}-1 \} $,
$$ 0 \longrightarrow K_{\ell_{m}-s}^{d_{m}} \longrightarrow K_{\ell_{m}-(s-1)}^{d_{m}} \xrightarrow{\emph{g}_{\ell_{\emph{m}}-\emph{s}}} \mathcal{O}_{D^{d_{m}}_{\ell_{m}-s}} \longrightarrow 0 $$
is a short exact sequence and $$ K_{\ell_{m}-s}^{d_{m}} \cong \Omega_{\mathbf{P}^n}^{1}(\log \{D_{1}^{d_{1}}, \ldots, D_{\ell_{1}}^{d_{1}}, \ldots \ldots, D_{1}^{d_{m}}, \ldots, D^{d_{m}}_{\ell_{m}-(s+1)}\}). $$
In particular
$$ K_{1}^{d_{m}} \cong \Omega_{\mathbf{P}^n}^{1}(\log \{D_{1}^{d_{1}}, \ldots, D_{\ell_{1}}^{d_{1}}, \ldots \ldots, D_{1}^{d_{m-1}}, \ldots, D^{d_{m-1}}_{\ell_{m-1}}\}) $$
and the smooth irreducible hypersurface $ D $ of degree $ d_{i} $ is unstable for $ K_{1}^{d_{m}} $. \\
If $ i = m-1 $, then $ D $ is a hypersurface of highest degree in the arrangement $ \mathcal{D}-\{D_{1}^{d_{m}}, \ldots, D_{\ell_{m}}^{d_{m}}\} $ and so, by repeating the computations of the first case of this proof, we get that there exists $ j \in \{1, \ldots, \ell_{m-1}\} $ such that $ D = D_{j}^{d_{m-1}} $. \\
If $ i = m-2 $, we apply the reduction technique to $ K_{1}^{d_{m}} $ and to the hypersurfaces $ \{D_{\ell_{m-1}}^{d_{m-1}}, \ldots, D_{1}^{d_{m-1}}\} $ and so on. \\
If $ i = 1 $, with this method $ D $ turns out to be unstable for the logarithmic bundle $ \Omega_{\mathbf{P}^n}^{1}(\log \{D_{1}^{d_{1}}, \ldots, D_{\ell_{1}}^{d_{1}}\}) $ and so, from Theorem \ref{T:samedegree} it follows that there exists $ j \in \{1, \ldots, \ell_{1}\} $ such that $ D = D^{d_{1}}_{j} $, which concludes the proof.
\end{proof}
We have the following:
\begin{cor}
If $ \ell_{i} \geq {{n+d_{i}} \choose {d_{i}}} + 3 $, for all $ i \in \{1, \ldots, m\} $, then the map 
$$ \mathcal{D} = \{ D_{1}^{d_{1}}, \ldots, D^{d_{1}}_{\ell_{1}}, \ldots \ldots, D_{1}^{d_{m}}, \ldots, D^{d_{m}}_{\ell_{m}} \} \longrightarrow \Omega_{\mathbf{P}^n}^{1}(\log \mathcal{D})  $$
is generically injective.
\end{cor}
\begin{rem}
Hypothesis $ 1. $ of Theorem \ref{multi-deg thm} implies (\ref{eq:inequality}). 
\end{rem}
\begin{rem}
We don't know if we can state a Torelli type theorem like \ref{multi-deg thm} without assuming 2. and 3. 
\end{rem}
In the case of arrangements with lines and conics in the projective plane, that is $ d_{1} = 1 $ and $ d_{2} = n = 2 $, hypothesis 1. of Theorem \ref{multi-deg thm} translates to $ \ell_{1} \geq 6 $ and $ \ell_{2} \geq 9 $.  In the next three sections we describe this kind of arrangements when $ \ell_{1} \in \{1,2,3\} $ and $  \ell_{2} = 1 $.

\section{A conic and a line}

Let $ \mathcal{D} = \{L, C\} $ be an arrangement with normal crossings in $ \mathbf{P}^2 $ consisting of a line $ L $ and a conic $ C $. Without loss of generality, we can assume $ L = \{ f_{1} = 0\}$ and $ C = \{f_{2} = 0 \} $, with $ f_{1}= x_{0} $ and $ f_{2} = { \sum_{i, j = 0}^{2}a_{i j} x_{i} x_{j}}, \,\, (a_{ij})_{0\leq i,j \leq 2} \in GL(2,\mathbf{C})  $, so that, by applying Gaussian elimination to the matrix of (\ref{eq:Anconas}), we can get the minimal resolution for $ \Omega_{\mathbf{P}^2}^{1}(\log \mathcal{D}) $
\begin{equation}\label{eq:minresline-conic}
0 \longrightarrow \mathcal{O}_{\mathbf{P}^2}(-2)  \buildrel \rm \it{M} \over \longrightarrow \mathcal{O}_{\mathbf{P}^2}(-1)^{2} \oplus \mathcal{O}_{\mathbf{P}^2} \longrightarrow \Omega_{\mathbf{P}^2}^{1}(\log \mathcal{D}) \longrightarrow 0
\end{equation}
with 
$$ M = \begin{pmatrix} 
2 \,\partial_{1}f_{2}  \cr 
2 \, \partial_{2}f_{2}  \cr 
-2 \, x_{0} \, \partial_{0}f_{2} \cr
\end{pmatrix}. $$
As a consequence we get that $ c_{1}(\Omega_{\mathbf{P}^2}^{1}(\log \mathcal{D})) = 0 $, $ c_{2}(\Omega_{\mathbf{P}^2}^{1}(\log \mathcal{D})) = 1 $ and, according to the Bohnhorst-Spindler criterion $($\cite{BS}$)$, that $ \Omega_{\mathbf{P}^2}^{1}(\log \mathcal{D}) $ is a semistable vector bundle over $ \mathbf{P}^2 $. 
\begin{thm0}\label{T:parM201}
Let $ \mathbf{M}_{\mathbf{P}^2}^{ss}(0,1) $ be the family of semistable rank $ 2 $ vector bundles $ E $  over $ \mathbf{P}^2 $ with minimal resolution 
$$ 0 \longrightarrow \mathcal{O}_{\mathbf{P}^2}(-2)  \xrightarrow{^{t}\begin{pmatrix} \ell_{1} & \ell_{2} & q \cr \end{pmatrix}} \mathcal{O}_{\mathbf{P}^2}(-1)^{2} \oplus \mathcal{O}_{\mathbf{P}^2} \longrightarrow E \longrightarrow 0 $$
where $ \ell_{1}, \ell_{2} \in H^{0}(\mathbf{P}^2, \mathcal{O}_{\mathbf{P}^2}(1)) $ and $ q \in H^{0}(\mathbf{P}^2, \mathcal{O}_{\mathbf{P}^n}(2)) $. Then the map
$$ \mathbf{M}_{\mathbf{P}^2}^{ss}(0,1) \buildrel \rm \pi_{2} \over \longrightarrow \mathbf{P}^2 $$
$$ E \longmapsto \{\ell_{1} = 0\} \cap \{\ell_{2} = 0\} $$
is an isomorphism. 
\end{thm0}
\begin{proof}
Let $ E $ and $ E' $ be two elements of  $ \mathbf{M}_{\mathbf{P}^2}^{ss}(0,1) $, defined, respectively, by $ \ell_{1}, \ell_{2}, q $ and $ \ell'_{1}, \ell'_{2}, q' $, as in the statement of Theorem \ref{T:parM201}. We have to prove that the intersection point of $ \ell_{1} $ and $ \ell_{2} $ coincides with the one of $ \ell'_{1} $ and $ \ell'_{2} $ if and only if $ E \cong E' $.  If the intersection point is the same, without loss of generality we can assume that $ \ell_{1} = \ell'_{1} = x_{0} $ and $ \ell_{2} = \ell'_{2} = x_{1} $. We remark that, for all $ x \in\mathbf{P}^2 $, $ E_{x} $ and $ E'_{x} $ are the cokernels of two rank $ 1 $ maps, in particular if $ x = [0,0,1] $ then $ q $ and $  q' $ have to contain the term $ x_{2}^{2} $. Thus, $ E \cong E' $ if and only if there exist $ g_{1}, g_{2} \in H^{0}(\mathbf{P}^2, \mathcal{O}_{\mathbf{P}^2}(1)) $ such that the following diagram commutes
$$ \mathcal{O}_{\mathbf{P}^2}(-2)  \xrightarrow{^{t}\begin{pmatrix} x_{0} & x_{1} & q \cr \end{pmatrix}} \mathcal{O}_{\mathbf{P}^2}(-1)^{2} \oplus \mathcal{O}_{\mathbf{P}^2} $$
$$ \quad\quad\quad\quad \begin{pmatrix} 1 \cr \end{pmatrix} \Big\downarrow \quad\quad\quad\quad\quad\quad\quad\quad\quad\quad\quad\quad \Big\downarrow \begin{pmatrix} 1 & 0 & 0 \cr 0 & 1 & 0 \cr g_{1} & g_{2} & 1 \cr  \end{pmatrix} $$
$$ \mathcal{O}_{\mathbf{P}^2}(-2)  \xrightarrow{^{t}\begin{pmatrix} x_{0} & x_{1} & q' \cr \end{pmatrix}} \mathcal{O}_{\mathbf{P}^2}(-1)^{2} \oplus \mathcal{O}_{\mathbf{P}^2} $$
which is equivalent to say that 
\begin{equation}\label{eq:lin}
q'- q = g_{1}x_{0}+g_{2}x_{1}. 
\end{equation}
Assume that 
$$ q = b_{00}x_{0}^{2}+b_{01}x_{0}x_{1}+b_{02}x_{0}x_{2}+b_{11}x_{1}^2+b_{12}x_{1}x_{2}+x_{2}^{2} $$
$$ q' = b'_{00}x_{0}^{2}+b'_{01}x_{0}x_{1}+b'_{02}x_{0}x_{2}+b'_{11}x_{1}^2+b'_{12}x_{1}x_{2}+x_{2}^{2}. $$
By using the identity principle for polynomials, we immediately get that
$$ g_{1} = (b'_{00}-b_{00})x_{0}+(b'_{01}-b_{01}-1)x_{1}+(b'_{02}-b_{02})x_{2} $$
$$ g_{2} = x_{0}+(b'_{11}-b_{11}-1)x_{1}+(b'_{12}-b_{12})x_{2} $$
solve (\ref{eq:lin}), which concludes the proof.
\end{proof}
\begin{rem}
Each $ E \in \mathbf{M}_{\mathbf{P}^2}^{ss}(0,1) $ is logarithmic for a line and a conic.
\end{rem}
\begin{rem}
Theorem \ref{T:parM201} asserts that $ \Omega_{\mathbf{P}^2}^{1}(\log \mathcal{D}) $ lives in $2$-dimensional space, while the number of parameters associated to a line and a conic with normal crossings is $ 7 $. So we can immediately conclude that arrangements like these are not of Torelli type.
\end{rem}
\noindent With the aid of the description given in Theorem \ref{T:parM201} and with the same notation as in the beginning of this section, we get the following result.
\begin{prop}
The point in $ \mathbf{P}^2 $ corresponding to $ \Omega_{\mathbf{P}^2}^{1}(\log \mathcal{D}) $ by means of $ \pi_{2} $ is the pole of the line $ L $ with respect to the conic $ C $.
\end{prop}
\begin{proof}
By applying Cramer's rule we get that the point in $ \mathbb{P}^{2} $ satisfying 
$$ \displaystyle{ \sum_{j = 0}^{2}a_{1 j} x_{j} } = \displaystyle{ \sum_{j = 0}^{2}a_{2 j} x_{j} } = 0 $$
is $ P = [ a_{12}^{2}-a_{11}a_{22}, a_{22}a_{01}-a_{02}a_{12}, a_{02}a_{11}-a_{12}a_{01} ] $. 
The polar line of $ P $ with respect to $ C $ is given by
$$ (a_{12}^{2}-a_{11}a_{22}, a_{22}a_{01}-a_{02}a_{12}, a_{02}a_{11}-a_{12}a_{01}) \begin{pmatrix} a_{00} & a_{01} & a_{02} \cr a_{01} & a_{11} & a_{12} \cr a_{02} & a_{12} & a_{22} \cr \end{pmatrix} \begin{pmatrix} x_{0} \cr x_{1} \cr x_{2} \cr \end{pmatrix} = 0  $$
which reduces to $ x_{0} = 0 $, that is to $ L $, as desired.
\end{proof}
\noindent We immediately get the following:
\begin{cor}\label{C:line-conic}
Let $ \mathcal{D} = \{L, C\} $ and $ \mathcal{D'} = \{L', C'\} $ be arrangements with normal crossings in $ \mathbf{P}^2 $ given by a line and a conic. Then 
$$ \Omega_{\mathbf{P}^2}^{1}(\log \mathcal{D}) \cong \Omega_{\mathbf{P}^2}^{1}(\log \mathcal{D'}) $$
if and only if the pole of $ L $ with respect to $ C $ coincides with the pole of $ L' $ with respect to $ C' $.
\end{cor}
\begin{figure}[h]
    \centering
    \includegraphics[width=65mm]{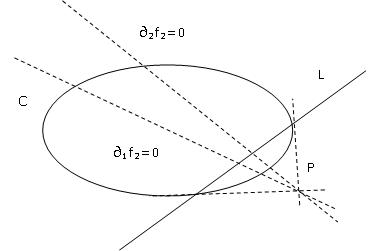}
    \caption{$ L $ is the polar line of $ P $ with respect to $ C $}
\end{figure}
\begin{rem}
These results can be extended in a natural way to the case of a multi-degree arrangement $ \mathcal{D} $ with normal crossings in $ \mathbf{P}^{n} $, $ n \geq 3 $, consisting of a hyperplane $ H $ and a smooth quadric $ Q $. In this setting $ \Omega_{\mathbf{P}^n}^{1}(\log \mathcal{D}) $ is no more semistable over $ \mathbf{P}^{n} $, but its isomorphism class is still described by the pole of $ H $ with respect to $ Q $, \cite{A}.  
\end{rem}

\section{A conic and two lines}

Let $ \mathcal{D} = \{L_{1}, L_{2}, C\} $ be an arrangement with normal crossings in $ \mathbf{P}^2 $, where $ L_{i} $, for $ i \in \{1,2\} $, is a line and $ C $ is a conic. We can assume that $ L_{1} = \{ f_{1} = 0 \} $, $ L_{2} = \{f_{2} = 0 \} $ and $ C = \{f_{3} = 0\} $ where $ f_{1} = x_{0} $, $ f_{2} = x_{1} $ and $ f_{3} = { \sum_{i, j = 0}^{2} a_{i j} x_{i} x_{j}}, \,\, (a_{ij})_{0\leq i,j \leq 2} \in GL(2,\mathbf{C}) $, so that, by means of (\ref{eq:Anconas}), $ \Omega_{\mathbf{P}^2}^{1}(\log \mathcal{D}) $ fits in the minimal resolution
\begin{equation}\label{eq:rescon2lines}
0 \longrightarrow \mathcal{O}_{\mathbf{P}^2}(-2) \buildrel \rm \it{M} \over \longrightarrow \mathcal{O}_{\mathbf{P}^2}(-1) \oplus \mathcal{O}_{\mathbf{P}^2}^{2} \longrightarrow \Omega_{\mathbf{P}^2}^{1}(\log \mathcal{D}) \longrightarrow 0
\end{equation}
where
$$  \it{M} = \begin{pmatrix}
2 \, \partial_{2}f_{3} \cr
-2 \, x_{0} \, \partial_{0}f_{3} \cr
-2 \, x_{1} \, \partial_{1}f_{3} \cr
\end{pmatrix}. $$
In particular, (\ref{eq:rescon2lines}) implies that the normalized bundle $ \Omega_{\mathbf{P}^2}^{1}(\log \mathcal{D})(-1) $ belongs to $ \mathbf{M}_{\mathbf{P}^2}(-1,2) $, the moduli space of rank $ 2 $ stable vector bundles over $ \mathbf{P}^2 $ with Chern classes $ c_{1} = -1 $ and $ c_{2} = 2 $. In the following result, which is likely to be known to experts, we give an interesting description of $ \mathbf{M}_{\mathbf{P}^2}(-1,2) $; in order to state it, we denote by $ \sigma_{2}(\nu_{2}(\mathbf{P}^2)) $ the $2$-secant variety of the image of the quadratic Veronese map $ \nu_{2} : \mathbf{P}^2 \longrightarrow \mathbf{P}^5 $.
\begin{thm0}\label{T:M(-1,2)}
$ \mathbf{M}_{\mathbf{P}^2}(-1,2) $ is isomorphic to $ \sigma_{2}(\nu_{2}(\mathbf{P}^2)) - \nu_{2}(\mathbf{P}^2) $, the projective space of symmetric matrices of order $ 3 $ and rank $ 2 $.
\end{thm0}
\begin{proof}
A vector bundle $ E $ lives in $ \mathbf{M}_{\mathbf{P}^2}(-1,2) $ if and only if it is endowed with a short exact sequence like
$$ 0 \longrightarrow \mathcal{O}_{\mathbf{P}^2}(-3) \xrightarrow{^{t}\begin{pmatrix} \ell_{1} & q_{1} & q_{2} \cr \end{pmatrix}} \mathcal{O}_{\mathbf{P}^2}(-2) \oplus \mathcal{O}_{\mathbf{P}^2}^{2}(-1) \longrightarrow E \longrightarrow 0 $$ 
where $ \ell_{1} \in H^{0}(\mathbf{P}^2, \mathcal{O}_{\mathbf{P}^2}(1)) $ and $ q_{1}, q_{2} \in H^{0}(\mathbf{P}^2, \mathcal{O}_{\mathbf{P}^2}(2)) $. \\
We note that $ E $ has a unique line $ L \subset \mathbf{P}^{2} $ such that $ H^{0}(L,E_{|_{L}}(-1)) \not= \{0\} $, known as \emph{jumping line} of $ E $, which is $ \{\ell_{1} = 0\} $. On this line, the linear series given by $ q_{1} $ and $ q_{2} $ has two distinct double points, which we denote by $ P_{1} $ and $ P_{2} $. Then the map given by 
$$ \mathbf{M}_{\mathbf{P}^2}(-1,2) \longrightarrow \sigma_{2}(\nu_{2}(\mathbf{P}^2)) - \nu_{2}(\mathbf{P}^2) $$
$$ E \longmapsto \{P_{1},P_{2}\} $$
is an isomorphism, which concludes the proof.
\end{proof}
\begin{rem}
Theorem \ref{T:M(-1,2)} implies that $ \Omega_{\mathbf{P}^2}^{1}(\log \mathcal{D})(-1) $ is characterized by $ 4 $ parameters, while $ \mathcal{D} $ needs $ 9 $ parameters to be described. So in this case $ \mathcal{D} $ is not a Torelli arrangement. 
\end{rem}
\begin{rem}\label{r:twopoints}
The jumping line of $ \Omega_{\mathbf{P}^2}^{1}(\log \mathcal{D}) $ is $ \{\partial_{2}f_{3} = 0\} $ and it is the polar line with respect to $ C $ of $  L_{1} \cap L_{2} = [0,0,1] $. Moreover, the linear series on this line is given by $ L_{1} \cup s_{2} $ and $ L_{2} \cup s_{1} $, where $ s_{2} $ is the polar line with respect to $ C $ of $ \{\partial_{2}f_{3} = 0\} \cap L_{2} = [a_{22}, 0, -a_{02}]$ and $ s_{1} $ is the polar line with respect to $ C $ of $ \{\partial_{2}f_{3} = 0\} \cap L_{1} = [0, a_{22}, -a_{12}] $, that is $  s_{2} = \{a_{22} \partial_{0}f_{3} - a_{02} \partial_{2}f_{3} = 0\} $ and $ s_{1} = \{a_{22} \partial_{1}f_{3} - a_{12} \partial_{2}f_{3} = 0\} $. The logarithmic bundle $ \Omega_{\mathbf{P}^2}^{1}(\log \mathcal{D}) $ corresponds to the two intersection points $ \{P_{1},P_{2}\} $ of $ C $ and $ \{\partial_{2}f_{3} = 0\} $. 
\end{rem}
\begin{figure}[h]
    \centering
    \includegraphics[width=100mm]{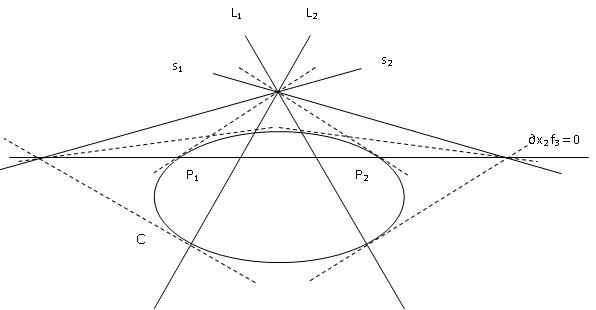}
    \caption{The points $ \{P_{1}, P_{2}\} $ associated to $ \Omega_{\mathbf{P}^2}^{1}(\log \mathcal{D}) $}
\end{figure}
\begin{cor}\label{C:2lines-conic}
Let $ \mathcal{D} = \{L_{1}, L_{2}, C\} $ and $ \mathcal{D'} = \{L'_{1}, L'_{2}, C'\} $ be arrangements with normal crossings in $ \mathbf{P}^2 $ consisting of two lines and a conic. Let $ \{P_{1},P_{2}\} $, resp. $ \{P'_{1},P'_{2}\} $, be the points in $ \mathbf{P}^2 $ associated to $ \Omega_{\mathbf{P}^2}^{1}(\log \mathcal{D}) $, resp. to $\Omega_{\mathbf{P}^2}^{1}(\log \mathcal{D'})$, in the sense of Remark \ref{r:twopoints}. Then 
$$ \Omega_{\mathbf{P}^2}^{1}(\log \mathcal{D}) \cong \Omega_{\mathbf{P}^2}^{1}(\log \mathcal{D'}) \Longleftrightarrow \{P_{1},P_{2}\} = \{P'_{1},P'_{2}\}. $$
\end{cor}

\section{A conic and three lines}

Let $ \mathcal{D} = \{L_{1},L_{2},L_{3}, C\} $ be an arrangement with normal crossings in $ \mathbf{P}^2 $ consisting of three lines and a conic, let us say that $ L_{1} = \{x_{0} = 0\} $, $ L_{2} = \{x_{1} =  0\} $, $ L_{3} = \{x_{2} =  0\} $ and $ C = \{f_{4} = 0\} $ where 
\begin{equation}\label{eq:f4}
f_{4} =  \displaystyle{ \sum_{i, j = 0}^{2}d_{i j} x_{i} x_{j}}, \,\, (d_{ij})_{0\leq i,j \leq 2} \in GL(2,\mathbf{C}) . 
\end{equation}  
\noindent In this case, starting from (\ref{eq:Anconas}), the minimal resolution for the logarithmic bundle turns out to be
\begin{equation}\label{eq:3lines-conic}
0 \longrightarrow \mathcal{O}_{\mathbf{P}^2}(-2) \buildrel \rm \it{M} \over \longrightarrow \mathcal{O}_{\mathbf{P}^2}^{3} \longrightarrow \Omega_{\mathbf{P}^2}^{1}(\log \mathcal{D}) \longrightarrow 0
\end{equation}
where
$$ \it{M} = \begin{pmatrix} 
-x_{0}\partial_{0}f_{4}  \cr 
-x_{1}\partial_{1}f_{4} \cr 
-x_{2}\partial_{2}f_{4} \cr 
\end{pmatrix}. $$
\noindent From (\ref{eq:3lines-conic}) we get that $ \Omega_{\mathbf{P}^2}^{1}(\log \mathcal{D})$ is stable and that its normalized bundle $ \Omega_{\mathbf{P}^2}^{1}(\log \mathcal{D})(-1) $ lives in the moduli space $ \mathbf{M}_{\mathbf{P}^2}(0,3) $, which has dimension $ 9 $, as we can see in \cite{OSS}. Since the number of parameters associated to three lines and a conic is $ 11 $, also in this case we can't get a Torelli type theorem. \\
\indent By using the second part of Theorem \ref{T:base}, we note that $ \Omega_{\mathbf{P}^2}^{1}(\log \mathcal{D})(-1) $ admits an exact sequence like the one for the logarithmic bundle of a smooth plane cubic curve. The link between these two vector bundles is explained in the following result:
\begin{thm0}\label{T:3lines-conic}
Let $ \mathcal{D} $ be the multi-degree arrangement with normal crossings on $ \mathbf{P}^{2} $ given by $ \{x_{0}x_{1}x_{2}f_{4} = 0\} $, where $ f_{4} $ is as in (\ref{eq:f4}). Then there exists $ \mathcal{D'} = \{D\} $, where $ D \subset \mathbf{P}^2 $ is a smooth cubic curve, such that 
$$ \Omega_{\mathbf{P}^2}^{1}(\log \mathcal{D}) \cong \Omega_{\mathbf{P}^2}^{1}(\log \mathcal{D'})(1). $$
\end{thm0}
\begin{proof}
Our aim is to find $ g \in H^{0}(\mathbf{P}^{2}, \mathcal{O}_{\mathbf{P}^{2}}(3)) $ such that, for all $ i \in \{0,1,2\} $, 
\begin{equation}\label{eq:cubica0}
\partial_{i}g = e_{0}^{i}(-x_{0}\partial_{0}f_{4})+e_{1}^{i}(-x_{1}\partial_{1}f_{4})+e_{2}^{i}(-x_{2}\partial_{2}f_{4})
\end{equation}
for certain $ e_{j}^{i} \in \mathbf{C} $. \\
By using Schwarz's theorem, from (\ref{eq:cubica0}) we get, for all $ i, h \in \{0,1,2\}, i \not= h $,
$$ \displaystyle{ \sum_{j = 0}^{2} e_{j}^{h}\partial_{i}(x_{j}\partial_{j}f_{4}) } = \displaystyle{ \sum_{j = 0}^{2} e_{j}^{i}\partial_{h}(x_{j}\partial_{j}f_{4}) }. $$
Let us denote by $ \{a_{uv}^{j}\} $ the coefficients of $ x_{j}\partial_{j}f_{4} $ for $ j \in \{0,1,2\} $; by using the identity principle for polynomials we get the following linear system of 9 equations with variables $ e_{j}^{i} $: for all $ i, u, v \in \{0,1,2\}, \, i \not= u $,
\begin{equation}\label{eq:system}
\displaystyle{ \sum_{j = 0}^{2} a_{uv}^{j} e_{j}^{i} } = \displaystyle{ \sum_{j = 0}^{2} a_{iv}^{j} e_{j}^{u} }. 
\end{equation}
Since $ a_{uv}^{j} $ depend on the coefficients of $ f_{4} $, the matrix of (\ref{eq:system}) is:
$$ H = \begin{pmatrix} 
d_{01} & d_{01} & 0 & -2d_{00} & 0 & 0 & 0 & 0 & 0 \cr 
0 & 2d_{11} & 0 & -d_{01} & -d_{01} & 0 & 0 & 0 & 0 \cr 
0 & d_{12} & d_{12} & -d_{02} & 0 & -d_{02} & 0 & 0 & 0 \cr 
d_{02} & 0 & d_{02} & 0 & 0 & 0 & -2d_{00} & 0 & 0 \cr
0 & d_{12} & d_{12} & 0 & 0 & 0 & -d_{01} & -d_{01} & 0 \cr
0 & 0 & 2d_{22} & 0 & 0 & 0 & -d_{02} & 0 & -d_{02} \cr
0 & 0 & 0 & d_{02} & 0 & d_{02} & -d_{01} & -d_{01} & 0 \cr
0 & 0 & 0 & 0 & d_{12} & d_{12} & 0 & -2d_{11} & 0 \cr
0 & 0 & 0 & 0 & 0 & 2d_{22} & 0 & -d_{12} & -d_{12} \cr
\end{pmatrix}. $$
By using Gaussian elimination, we get that $ rank \, H \leq 8 $. So, let us assume that $ \overline{e}_{j}^{i} $ is a solution of our system, we need a cubic polynomial $ g $ such that conditions in (\ref{eq:cubica0}) are satisfied. Let us integrate with respect to $ x_{0} $ the equation (\ref{eq:cubica0}) with $ i = 0 $, we get
\begin{equation}\label{eq:cubicax0}
g(x_{0},x_{1},x_{2})=-2\overline{e}_{0}^{0}x_{0}^{2}\left({x_{0} \over 3}a_{00}^{0}+{x_{1} \over 2}a_{01}^{0}+{x_{2} \over 2}a_{02}^{0} \right)+ h(x_{1},x_{2}) +
\end{equation}
$$ -2x_{0}\left[ \overline{e}_{1}^{0}\left({x_{0}x_{1} \over 2}a_{01}^{0}+x_{1}^2a_{11}^{0}+x_{1}x_{2}a_{12}^{0} \right) + \overline{e}_{2}^{0}\left({x_{0}x_{2} \over 2}a_{02}^{0}+x_{1}x_{2}a_{12}^{0}+x_{2}^2a_{22}^{0} \right) \right] $$
where $ h $ is a function to be determined. 
If we compute $ \partial_{1}g $ from (\ref{eq:cubicax0}), we substitute it in (\ref{eq:cubica0}) with $ i = 1 $ and we integrate with respect to $ x_{1} $ we get
\begin{equation}\label{eq:cubicax1}
h(x_{1},x_{2}) = x_{0}x_{1}\left[\overline{e}_{0}^{0}x_{0}a_{01}^{0}+2\overline{e}_{1}^{0}\left({x_{0} \over 2}a_{01}^{0}+x_{1}a_{11}^{0}+x_{2}a_{12}^{0} \right) \right]+ 
\end{equation}
$$ +x_{0}x_{1}\left[2\overline{e}_{2}^{0}x_{2}a_{12}^{0}-\overline{e}_{0}^{1}\left(2x_{0}a_{00}^{0}+x_{1}a_{01}^{0}+2x_{2}a_{02}^{0} \right)\right]+  i(x_{2}) + $$
$$  -\overline{e}_{1}^{1}x_{1}^{2}\left(x_{0}a_{01}^{0}+2{x_{1} \over 3}a_{11}^{0}+x_{2}a_{12}^{0} \right) -x_{1}x_{2}\overline{e}_{2}^{1}\left(2x_{0}a_{02}^{0}+x_{1}a_{12}^{0}+2x_{2}a_{22}^{0} \right)  $$
where we have to determine the function $ i $.
Finally, if we compare $ \partial_{2}g $ from (\ref{eq:cubicax0}) with (\ref{eq:cubica0}) for $ i = 2 $, using also  (\ref{eq:cubicax1}) and we integrate with respect to $ x_{2} $, we can find explicitly $ i $, so that the required polynomial is
$$ g(x_{0},x_{1},x_{2}) = - {2 \over 3}\overline{e}_{0}^{0}a_{00}^{0}x_{0}^3-2\overline{e}_{0}^{1}a_{00}^{0}x_{0}^2x_{1}-2\overline{e}_{1}^{0}a_{11}^{0}x_{0}x_{1}^2- {2 \over 3}\overline{e}_{1}^{1}a_{11}^{0}x_{1}^3-2\overline{e}_{2}^{0}a_{22}x_{0}x_{2}^2 + $$
$$ -2\overline{e}_{0}^{2}a_{00}^{0}x_{0}^2x_{2}-2(\overline{e}_{1}^{0}+\overline{e}_{2}^{0})a_{12}^{0}x_{0}x_{1}x_{2}-2\overline{e}_{1}^{2}a_{11}^{0}x_{1}^2x_{2}
-2\overline{e}_{2}^{1}a_{22}^{0}x_{1}x_{2}^2- {2 \over 3}\overline{e}_{2}^{2}a_{22}^{0}x_{2}^3. $$
\end{proof}
\begin{rem}
The proof of Theorem \ref{T:3lines-conic} implies also Hermite's Theorem (1868), which asserts that a net of conics can be regarded as the net of the polar conics with respect to a given cubic curve (see \cite{EC}, book III, chapter III, section 29).
\end{rem}
\begin{rem}
If we require that $ \partial_{i}g = x_{i}\partial_{i}f_{4} $, for all $ i \in \{0,1,2\} $, then 
\begin{equation}\label{eq:fermatcubic}
g(x_{0},x_{1},x_{2}) = {2 \over 3}\left(a_{00}^{0}x_{0}^3+a_{11}^{0}x_{1}^3+a_{22}^{0}x_{2}^3\right), 
\end{equation}
provided that the conic is given by $ f_{4}(x_{0},x_{1},x_{2}) = d_{00}x_{0}^2+d_{11}x_{1}^2+d_{22}x_{2}^2. $
So, let $ \mathcal{D}=\{x_{0}x_{1}x_{2}f_{4}=0\} $ and $ \mathcal{D}' =\{x_{0}x_{1}x_{2}f'_{4}=0\} $ be two arrangements with normal crossings in $ \mathbf{P}^2 $ each of which with a conic given by a diagonalized quadratic form. $ \mathcal{D} $ and $ \mathcal{D}' $ correspond to a logarithmic bundle which is isomorphic to the logarithmic bundle of a smooth cubic like the one of (\ref{eq:fermatcubic}). Since in \cite{UY1} it is proved that two smooth cubics which are both Fermat yield isomorphic logarithmic bundles, then $ \Omega_{\mathbf{P}^2}^{1}(\log \mathcal{D}) \cong \Omega_{\mathbf{P}^2}^{1}(\log \mathcal{D'}) $.
\end{rem}
\begin{rem}
Although we know that a multi-degree arrangement with three lines and a conic, because of parameters computations, isn't Torelli  and that Theorem \ref{T:3lines-conic} holds, in this case the problem of determining the fiber of (\ref{eq:Torellimap}) is still open.
\end{rem}

\section{Arrangements with few conics}

Let $ \mathcal{D} = \{C_{1}, \ldots, C_{\ell}\} $ be an arrangement of $ \ell \in \{4, \ldots, 8 \} $ conics with normal crossings on $ \mathbf{P}^{2} $.\\
\indent Let $ \mathbf{F}^{2}_{5} = \{(x,y) \in \mathbf{P}^{2} \times \mathbf{P}_{5} \, | \, x \in C_{y}\} $ be the \emph{incidence variety} point-conic in $ \mathbf{P}^{2} \times \mathbf{P}_{5} $, where $ C_{y} \subset \mathbf{P}^{2} $ denotes the conic defined by the point $ y \in \mathbf{P}_{5} $ with the \emph{Veronese} correspondence and let $ \overline{\alpha} $, $ \overline{\beta} $  the restrictions to $ \mathbf{F}^{2}_{5} $ of the usual projections $ \alpha $ and $ \beta $:
\\
$$ \mathbf{F}^{2}_{5} \subset \mathbf{P}^{2} \times \mathbf{P}_{5} $$
$$ \buildrel \rm {\overline{\alpha}} \over \swarrow \quad\quad\quad \buildrel \rm {\overline{\beta}} \over \searrow \quad\quad\quad\quad $$ 
$$ \mathbf{P}^2  \quad\quad\quad\quad\quad \mathbf{P}_5 \quad\quad\quad\quad  $$
\begin{rem}\label{r:unstcon2}
Let $ UC(\Omega_{\mathbf{P}^2}^{1}(\log \mathcal{D})) $ be the set of unstable conics of $ \Omega_{\mathbf{P}^2}^{1}(\log \mathcal{D}) $, in the sense of Definition \ref{D:hypunst}. $ UC(\Omega_{\mathbf{P}^2}^{1}(\log \mathcal{D})) $ coincides with the support of the first direct image sheaf $ R^1(\overline{\beta}_{\ast}\overline{\alpha}^{\ast}\Omega_{\mathbf{P}^2}^{1}(\log \mathcal{D})(-1)) $: indeed, for all $ y\in \mathbf{P}_{5} $, 
$$ R^1(\overline{\beta}_{\ast}\overline{\alpha}^{\ast}\Omega_{\mathbf{P}^2}^{1}(\log \mathcal{D})(-1))_{y} = H^{1}(\overline{\beta}^{-1}(y), \overline{\alpha}^{\ast}\Omega_{\mathbf{P}^2}^{1}(\log \mathcal{D})(-1)_{|_{\overline{\beta}^{-1}(y)}}) = $$
$$ = H^{1}(C_{y}, \Omega_{\mathbf{P}^2}^{1}(\log \mathcal{D})(-1)_{|_{C_{y}}}) = H^{0}(C_{y}, {\Omega_{\mathbf{P}^2}^{1}(\log \mathcal{D})}^{\vee}_{|_{C_{y}}})^{\vee}, $$
where the last inequality follows from Serre's duality.
\end{rem}
\noindent So, let tensor with $ \mathcal{O}_{\mathbf{P}^2}(-1) $ the exact sequence (\ref{eq:Anconas}) where $ n =2, d_{i} = 2 $ and let apply the functor $ \overline{\beta}_{\ast}\overline{\alpha}^{\ast} $, we get: 
\begin{equation}\label{eq:dirimage}
0 \rightarrow R^{0}\overline{\beta}_{\ast}\overline{\alpha}^{\ast}( \mathcal{O}_{\mathbf{P}^2}(-3))^{\ell} \rightarrow R^{0}\overline{\beta}_{\ast}\overline{\alpha}^{\ast}( \mathcal{O}_{\mathbf{P}^2}(-2)^{3} \oplus \mathcal{O}_{\mathbf{P}^2}(-1)^{\ell-1}) \rightarrow 
\end{equation}
$$ \rightarrow R^{0}\overline{\beta}_{\ast}\overline{\alpha}^{\ast}(\Omega_{\mathbf{P}^2}^{1}(\log \mathcal{D})(-1)) \rightarrow R^{1}\overline{\beta}_{\ast}\overline{\alpha}^{\ast}( \mathcal{O}_{\mathbf{P}^2}(-3))^{\ell} \rightarrow $$
$$ \rightarrow R^{1}\overline{\beta}_{\ast}\overline{\alpha}^{\ast}( \mathcal{O}_{\mathbf{P}^2}(-2)^{\ell-1} \oplus \mathcal{O}_{\mathbf{P}^2}(-1)^{3}) \rightarrow R^{1}\overline{\beta}_{\ast}\overline{\alpha}^{\ast}(\Omega_{\mathbf{P}^2}^{1}(\log \mathcal{D})(-1)) \rightarrow 0.  $$
In order to determine the terms in (\ref{eq:dirimage}) we consider
$$ 0 \longrightarrow \mathcal{O}_{\mathbf{P}^{2} \times \mathbf{P}_{5}}(-2, -1) \longrightarrow \mathcal{O}_{\mathbf{P}^{2} \times \mathbf{P}_{5}} \longrightarrow \mathcal{O}_{{\mathbf{F}_{5}^{2}}} \longrightarrow 0, $$
we do the tensor product with $ \alpha^{\ast}(\mathcal{O}_{\mathbf{P}^{2}}(t)) $, where $ t \in \{-1, -2, -3\} $ and we apply the functor $ \beta_{\ast} $. In this way (\ref{eq:dirimage}) becomes
\begin{equation}\label{eq:unstconseq} 0 \longrightarrow  R^{0}\overline{\beta}_{\ast}\overline{\alpha}^{\ast}(\Omega_{\mathbf{P}^2}^{1}(\log \mathcal{D})(-1)) \longrightarrow (\Omega^{1}_{\mathbf{P}_{5}})^{\ell} \buildrel \rm \it{F} \over \longrightarrow 
\end{equation}
$$ \buildrel \rm \it{F} \over \longrightarrow (\mathcal{O}_{\mathbf{P}_{5}}(-1)^{3})^{3} \oplus \mathcal{O}_{\mathbf{P}_{5}}(-1)^{\ell-1} \longrightarrow R^{1}\overline{\beta}_{\ast}\overline{\alpha}^{\ast}(\Omega_{\mathbf{P}^2}^{1}(\log \mathcal{D})(-1)) \longrightarrow 0. $$
\begin{rem}
In order to investigate $ UC(\Omega_{\mathbf{P}^2}^{1}(\log \mathcal{D})) $, it suffices to study the cokernel of the map $ F $ appearing in (\ref{eq:unstconseq}). 
\end{rem}
\begin{rem}
More generally, all the previous arguments can be applied to a vector bundle $ E $ fitting in an exact sequence like the one of $ \Omega_{\mathbf{P}^2}^{1}(\log \mathcal{D}) $.
\end{rem} 
Now, let us assume that $ \ell = 4 $. In what follows, by using Macaulay2 software system, we produce $ \mathcal{D}_{0} = \{ C_{0, \,1}, C_{0, \,2}, C_{0, \,3}, C_{0, \, 4} \} $ such that 
\begin{equation}\label{eq:equality}
UC(\Omega_{\mathbf{P}^2}^{1}(\log \mathcal{D}_{0})) =  \{ C_{0, \,1}, C_{0, \,2}, C_{0, \,3}, C_{0, \, 4} \}.
\end{equation}
\noindent{\bf Example 1}\\
\noindent $ \mathcal{D}_{0} $ is made of four smooth random conics with normal crossings:
$$ C_{0,1}: 42x_{0}^{2}-50x_{0}x_{1}+9x_{1}^{2}+39x_{0}x_{2}-15x_{1}x_{2}-22x_{2}^{2} = 0, $$
$$ C_{0,2}: 50x_{0}^{2}+45x_{0}x_{1}-39x_{1}^{2}-29x_{0}x_{2}+30x_{1}x_{2}+19x_{2}^{2} = 0, $$
$$ C_{0,3}: -38x_{0}^{2}+2x_{0}x_{1}-36x_{1}^{2}-4x_{0}x_{2}-16x_{1}x_{2}-6x_{2}^{2} = 0, $$
$$ C_{0,4}: -32x_{0}^{2}+31x_{0}x_{1}-38x_{1}^{2}-32x_{0}x_{2}+31x_{1}x_{2}+24x_{2}^{2} = 0. $$
By multiplying the four polynomials defining the conics, we get the polynomial $ f \in k[x_{0},x_{1},x_{2}]_{8} = R_{8} $ associated to $ \mathcal{D}_{0} $, where $ k $ is the field $ \mathbf{Z}_{101} $. According to Definition \ref{d:Gaussmap}, we consider the kernel $ E $ of the \emph{Gauss map} and we construct the matrix $ M \in M_{6,4}(R) $ associated to the module defining $ \Omega_{\mathbf{P}^2}^{1}(\log \mathcal{D}_{0}) $. Then we determine the elements of $ UC(\Omega_{\mathbf{P}^2}^{1}(\log \mathcal{D}_{0})) $: as we can see in Remark \ref{r:unstcon2}, $ UC(\Omega_{\mathbf{P}^2}^{1}(\log \mathcal{D}_{0})) $ is the zero locus of the order $ 4 $ minors of the matrix $ Z $, whose cokernel is equal to the cokernel of $ F $. In particular, posing $ T = k[y_{0}, \ldots, y_{5}] $, $ Z \in M_{4,12}( T) $ is the product of $ C \in M_{4,24}( T) $ and $ B \in M_{24,12}( T) $, where $ C $ is the matrix of variables needed to get $ \Omega_{\mathbf{P}^5}^{1} $ and $ B $ is the syzygy matrix of $ A \in M_{12,24}( T) $ whose entries are the coefficients of the polynomials in $ M $. The ideal $ J $ generated by the $ 4 \times 4 $ minors of $ Z $ has dimension $ 1 $ and degree $ 4 $, from which (\ref{eq:equality}) follows. \\
\noindent This is the script of our algorithm.
\begin{verbatim}
k=ZZ/101
R=k[x_0..x_2]
ran=random(R^{1:0},R^{4:-2})
f=1_R; for t from 0 to rank source ran-1 do f=f*(ran_(0,t))
E=ker map(R^{1:-1+(degree f)_0},R^{3:0},diff(vars R,f))
M=(res dual E).dd_1
T=k[y_0..y_5]
coe=(M,k,i,j)->diff(symmetricPower(k,vars R),transpose(symmetricPower
(k-2,vars R))*submatrix(M,{i},{j}))
coe2=(M,i,j)->diff(transpose(vars(R))*submatrix(M,{i},{j}),symmetric
Power(2,vars R))
expa=(M,k)->matrix table(rank target M,rank source M,(i,j)->coe(M,k,
i,j))
expa2=(M)->matrix table(rank target M,rank source M,(i,j)->coe2(M,i,
j))
A=sub(matrix(expa(submatrix(M,{0..2},{0..3}),2), expa2(submatrix(M,
{3..5},{0..3}))),T)
B=syz A
C=(id_(T^{4:0}))**(vars T)
Z=C*B
J=minors(4,Z)
dim J
degree J
\end{verbatim}

\begin{rem}
The previous algorithm can be performed for all $ \ell $. In particular, if $ \ell=5 $ then we can get another example such that the unstable conics of the logarithmic bundle coincide with the conics of the arrangement.
\end{rem}
Starting from the previous example, we can prove the following:
\begin{thm0}\label{T:molteconiche} If $ \ell \geq 4 $, then the map 
$$ \mathcal{D} = \{ C_{1}, \ldots, C_{\ell} \} \longrightarrow \Omega_{\mathbf{P}^2}^{1}(\log \mathcal{D})  $$
is generically injective. 
\end{thm0}
\begin{proof}
First, let us assume that $ \ell = 4 $. Let us consider the \emph{incidence variety} $ W = \{ (\mathcal{D}, C) \in (\mathbf{P}_5 \times \mathbf{P}_5 \times \mathbf{P}_5 \times \mathbf{P}_5) \times \mathbf{P}_5 \, | \, C \in UC(\Omega_{\mathbf{P}^2}^{1}(\log \mathcal{D}))  \} $ and let $ \overline{a} $, $ \overline{b} $ be the restrictions to $ W $ of the projection morphisms, respectively,  from $ (\mathbf{P}_5 \times \mathbf{P}_5 \times \mathbf{P}_5 \times \mathbf{P}_5) $ and $\mathbf{P}_5 $. \\
From the previous example we have that $ \overline{a}^{-1}(\mathcal{D}_{0}) = \mathcal{D}_{0} $. So, for all arrangements $ \mathcal{D} \in \mathbf{P}_5 \times \mathbf{P}_5 \times \mathbf{P}_5 \times \mathbf{P}_5 $, $ dim \,  \overline{a}^{-1}(\mathcal{D}) \geq 0 $ and $ h^{0}(\overline{a}^{-1}(\mathcal{D}), \mathcal{O}_{W}) = length \,  \overline{a}^{-1}(\mathcal{D}) \geq 4. $   
To conclude the proof, it suffices to show that there exists $ V \subset \mathbf{P}_5 \times \mathbf{P}_5 \times \mathbf{P}_5 \times \mathbf{P}_5 $ open such that, for all $ \mathcal{D} \in V $ 
\begin{equation}\label{eq:0dim}
dim \, \overline{a}^{-1}(\mathcal{D}) = 0,
\end{equation}
\begin{equation}\label{eq:lenght4}
length \, \overline{a}^{-1}(\mathcal{D}) = 4.
\end{equation}
We remark that the dimension $ d $ of the fiber given by the morphism $ \overline{a} $ has the upper semicontinuity property (\cite{Mum}, chapter 1, section 8, corollary 3), which implies that $ \{w \in W \, | \, d(w)\geq 1 \} $  
is a closed subset in $ W $. So the set $ V_{1} = \overline{a}(\{w \in W \, | \, d(w) \leq 0\}) = \overline{a}(\{w \in W \, | \, d(w) = 0\}) $
is open in $ \mathbf{P}_5 \times \mathbf{P}_5 \times \mathbf{P}_5 \times \mathbf{P}_5 $.
By using the upper semicontinuity of the length of the fiber given by the morphism $ \overline{a} $ (this fact is a consequence of theorem $ 12.8 $, chapter $ 3 $ of \cite{H}; this theorem holds with the hypothesis of flatness, in our case we have the generic flatness) we get that the set $ \{ \mathcal{D} \in \mathbf{P}_5 \times \mathbf{P}_5 \times \mathbf{P}_5 \times \mathbf{P}_5 \, | \, length \,\overline{a}^{-1}(\mathcal{D}) \geq 5 \} $ is closed in $ \mathbf{P}_5 \times \mathbf{P}_5 \times \mathbf{P}_5 \times \mathbf{P}_5 $. As above, the set $ V_{2} = \{\mathcal{D} \in \mathbf{P}_5 \times \mathbf{P}_5 \times \mathbf{P}_5 \times \mathbf{P}_5 \, | \, length \,\overline{a}^{-1}(\mathcal{D}) \leq 4\} =  \{\mathcal{D} \in \mathbf{P}_5 \times \mathbf{P}_5 \times \mathbf{P}_5 \times \mathbf{P}_5 \, | \, length \,\overline{a}^{-1}(\mathcal{D}) = 4\}  $ is open in $ \mathbf{P}_5 \times \mathbf{P}_5 \times \mathbf{P}_5 \times \mathbf{P}_5. $ The points of the open set $ V = V_{1} \cap V_{2} $ satisfy the required properties (\ref{eq:0dim}) and (\ref{eq:lenght4}). \\
Now, if $ \ell \geq 5 $, then we can apply the \emph{reduction technique}, performed in the proof of Theorem \ref{multi-deg thm}, to $ \Omega_{\mathbf{P}^2}^{1}(\log \mathcal{D}) $ and to the conics of $ \mathcal{D} $: at each step we get a logarithmic bundle of a conic-arrangement with one component less, till we reduce to the case of four conics, studied above.   
\end{proof}

Finally we discuss the case of $ \ell = 3 $. \\
Let $ \mathcal{D} = \{C_{1}, C_{2}, C_{3}\} $ be an arrangement of conics with normal crossings on $ \mathbf{P}^{2} $. Let us start by analyzing $ UC(\Omega_{\mathbf{P}^2}^{1}(\log \mathcal{D})) $. In order to do that, let us consider the exact sequence (\ref{eq:unstconseq}) with $ \ell = 3 $: $ UC(\Omega_{\mathbf{P}^2}^{1}(\log \mathcal{D})) $, the support of $ R^{1}\overline{\beta}_{\ast}\overline{\alpha}^{\ast}(\Omega_{\mathbf{P}^2}^{1}(\log \mathcal{D})(-1)) $, is the maximal degeneration locus of the morphism $ (\Omega^{1}_{\mathbf{P}_{5}})^{3} \buildrel \rm F \over \longrightarrow (\mathcal{O}_{\mathbf{P}_{5}}(-1)^{3})^{3} \oplus \mathcal{O}_{\mathbf{P}_{5}}(-1)^{2} $, i.e. it coincides with the scheme $ D_{10}(F) = \{ y \in \mathbf{P}_{5} \, | \, rank(F_{y}) \leq 10  \} $, which, according to \cite{O0}, has expected codimension $ 5 $ in $ \mathbf{P}_{5} $ (we note that the computation of the expected codimension is meaningless when $ \ell \geq 4 $). If this is the case, the number of points in $ D_{10}(F) $ is determined by Porteous' formula: 
\begin{equation}\label{eq:Porteous}
[D_{10}(F)] = det [ 
c_{1-i+j}((((\mathcal{O}_{\mathbf{P}_{5}}(-1)^{3})^{3} \oplus \mathcal{O}_{\mathbf{P}_{5}}(-1)^{2}) - (\Omega^{1}_{\mathbf{P}_{5}})^{3})],  
\end{equation}
where $ 1 \leq i,j \leq 5 $. The generic entry of the matrix (\ref{eq:Porteous}) is the coefficient of the term of degree $ (1-i+j) $ in the formal series in one variable coming from the quotient of the Chern polynomials of $ ((\mathcal{O}_{\mathbf{P}_{5}}(-1)^{3})^{3} \oplus \mathcal{O}_{\mathbf{P}_{5}}(-1)^{2}) $ and $ (\Omega^{1}_{\mathbf{P}_{5}})^{3} $. Thus $ [D_{10}(F)] = 21. $ More generally, we get the following:
\begin{prop}\label{p:21conics}
Let $ E $ be a vector bundle over $ \mathbf{P}^{2} $ such that
$$ 0 \longrightarrow \mathcal{O}_{\mathbf{P}^{2}}(-2)^{3} \longrightarrow \mathcal{O}_{\mathbf{P}^{2}}(-1)^{3} \oplus \mathcal{O}_{\mathbf{P}^{2}}^{2} \longrightarrow E \longrightarrow 0 $$
is exact and let $ UC(E) $ be the set of unstable conics of $ E $, in the sense of (\ref{eq:hypunst}).
$ UC(E) $ is expected to be a 0-dimensional scheme of $ \mathbf{P}_{5} $ with 21 points.
\end{prop}
\begin{rem}\label{r:nounstcon}
If we apply the algorithm performed in Example 1 of this section in the case of $ \ell $ = 3, we can find some arrangements $ \mathcal{D} $ such that $ UC(\Omega^{1}_{\mathbf{P}^{2}}(log \mathcal{D})) $ satisfies the expected properties of Proposition \ref{p:21conics}. Indeed, according to the notations introduced in such algorithm, the variety in $ \mathbf{P}_{5} $ defined by the ideal $ J $ has 21 distinct points, which, in terms of the quadratic Veronese embedding of the projective plane, correspond to smooth conics in $ \mathbf{P}^{2} $. Between these $ 21 $ points, 3 correspond to the component of $ \mathcal{D} $ and the remaining 18 belong to a net quadrics in $ \mathbf{P}_{5} $, whose base locus is a $ K$3-surface with 12 singular points, that don't seem to be related to the 18 conics we are interested in. The explicit determination of such 18 points or, equivalently, of a primary decomposition of $ J $ saturated with the ideals defining the conics of $ \mathcal{D} $ as points in $ \mathbf{P}_{5} $, would be interesting to solve the Torelli problem in this case, but, at the moment, it seems to be hard, also with a computer.           
\end{rem}
According to Remark \ref{r:nounstcon}, instead of studying $ UC(\Omega_{\mathbf{P}^2}^{1}(\log \mathcal{D})) $, we can focus on $ UL(\Omega_{\mathbf{P}^2}^{1}(\log \mathcal{D})) $,  the set of \emph{unstable lines} of $ \Omega_{\mathbf{P}^2}^{1}(\log \mathcal{D}) $ in the sense of Definition \ref{D:hypunst}. Let $ \mathbf{F}^{2}_{2} $ be the \emph{incidence variety} point-line in $ \mathbf{P}^2 \times \mathbf{P}_{2} $, i.e.
\begin{equation}
\mathbf{F}^{2}_{2} = \{(x,y) \in  \mathbf{P}^2 \times \mathbf{P}_{2} \, | \, x \in L_{y} \}
\end{equation}
where $ L_{y} \subset \mathbf{P}^2 $ is the line defined by $ y \in \mathbf{P}_{2} $ and let $ \overline{p} $, $ \overline{q} $ be, respectively, the restrictions to $ \mathbf{F}^{2}_{2} $ of the projection maps $ p $, $ q $ as in the following diagram:
$$ \mathbf{F}^{2}_{2} \subset \mathbf{P}^{2} \times \mathbf{P}_{2} $$
$$ \quad\quad \buildrel \rm {\overline{p}} \over \swarrow \quad\quad\quad \buildrel \rm {\overline{q}} \over \searrow \quad\quad\quad\quad\quad\quad\quad $$ 
$$ \mathbf{P}^2  \quad\quad\quad\quad\quad \mathbf{P}_2 \quad\quad\quad\quad\quad  $$
We remark that $ UL(\Omega_{\mathbf{P}^2}^{1}(\log \mathcal{D})) $, as a subset of $ \mathbf{P}_{2} $, is the support of $ R^1(\overline{q}_{\ast}\overline{p}^{\ast}\Omega_{\mathbf{P}^2}^{1}(\log \mathcal{D})(-2)) $. Namely, if $ y\in \mathbf{P}_{2} $ then we have that 
$$ R^1(\overline{q}_{\ast}\overline{p}^{\ast}\Omega_{\mathbf{P}^2}^{1}(\log \mathcal{D})(-2))_{y} = H^{1}(\overline{q}^{-1}(y), \overline{p}^{\ast}\Omega_{\mathbf{P}^2}^{1}(\log \mathcal{D})(-2)_{|_{\overline{q}^{-1}(y)}}) = $$
$$ = H^{1}(L_{y}, \Omega_{\mathbf{P}^2}^{1}(\log \mathcal{D})(-2)_{|_{L_{y}}}) = H^{0}(L_{y}, {\Omega_{\mathbf{P}^2}^{1}(\log \mathcal{D})}^{\vee}_{|_{L_{y}}})^{\vee}, $$
where the last equality follows from Serre's duality. In order to study this support, we apply the functor $ \overline{q}_{\ast} \overline{p}^{\ast} $ to the exact sequence (\ref{eq:Anconas}) in the case of three conics twisted by $ -2 $ and we get
\begin{equation}\label{eq:directimagelines} 
0 \longrightarrow R^{0}\overline{q}_{\ast} \overline{p}^{\ast} (\mathcal{O}_{\mathbf{P}^2}(-4)^{3}) \longrightarrow R^{0}\overline{q}_{\ast} \overline{p}^{\ast} (\mathcal{O}_{\mathbf{P}^2}(-3)^{3} \oplus \mathcal{O}_{\mathbf{P}^2}(-2)^{2}) \longrightarrow 
\end{equation}
$$ \longrightarrow R^{0}\overline{q}_{\ast} \overline{p}^{\ast} (\Omega_{\mathbf{P}^2}^{1}(\log \mathcal{D})(-2)) \longrightarrow R^{1}\overline{q}_{\ast} \overline{p}^{\ast} (\mathcal{O}_{\mathbf{P}^2}(-4)^{3}) \longrightarrow  \,\, $$
$$ \longrightarrow R^{1}\overline{q}_{\ast} \overline{p}^{\ast} (\mathcal{O}_{\mathbf{P}^2}(-3)^{3} \oplus \mathcal{O}_{\mathbf{P}^2}(-2)^{2}) \longrightarrow R^{1}\overline{q}_{\ast} \overline{p}^{\ast} (\Omega_{\mathbf{P}^2}^{1}(\log \mathcal{D})(-2)) \longrightarrow 0.  $$
Our aim is to describe the terms of (\ref{eq:directimagelines}). So we tensor   
$$ 0 \longrightarrow \mathcal{O}_{{\mathbf{P}^2} \times {\mathbf{P}_{2}}} (-1,-1) \longrightarrow \mathcal{O}_{{\mathbf{P}^2} \times {\mathbf{P}_{2}}} \longrightarrow \mathcal{O}_{\mathbf{F}^{2}_{2}} \longrightarrow 0 $$
with $ p^{\ast} \mathcal{O}_{\mathbf{P}}^2 (t) $, where $ t \in \{-4,-3,-2\} $ and then we apply $ q_{\ast} $.
By using Serre's duality and the Poincar\'e-Euler sequence, (\ref{eq:directimagelines}) turns out to be 
\begin{equation}\label{eq:supportsequence}
0 \longrightarrow  R^{0}\overline{q}_{\ast}\overline{p}^{\ast}(\Omega_{\mathbf{P}^2}^{1}(\log \mathcal{D})(-2)) \longrightarrow R^{1}\overline{q}_{\ast}\overline{p}^{\ast}(\mathcal{O}_{\mathbf{P}^2}(-4)^3)  \buildrel \rm G \over \longrightarrow
\end{equation}
$$ \quad\quad \buildrel \rm G \over \longrightarrow (\Omega_{\mathbf{P}_{2}}^{1})^{3} \oplus \mathcal{O}_{\mathbf{P}_{2}}(-1)^{2} \longrightarrow R^{1}\overline{q}_{\ast}\overline{p}^{\ast}(\Omega_{\mathbf{P}^2}^{1}(\log \mathcal{D})(-2)) \longrightarrow 0 $$
where $ R^{1}\overline{q}_{\ast}\overline{p}^{\ast}(\mathcal{O}_{\mathbf{P}^2}(-4)) $ fits in the exact sequence
$$ 0 \longrightarrow R^{1}\overline{q}_{\ast}\overline{p}^{\ast}(\mathcal{O}_{\mathbf{P}^2}(-4)) \longrightarrow \mathcal{O}_{\mathbf{P}_{2}}(-1)^{6} \longrightarrow \mathcal{O}_{\mathbf{P}_{2}}^{3} \longrightarrow 0, $$
and it has rank $ 3 $ over $ \mathbf{P}^{2} $.
The support of $ R^1(\overline{q}_{\ast}\overline{p}^{\ast}\Omega_{\mathbf{P}^2}^{1}(\log \mathcal{D})(-2)) $ is the maximal degeneration locus of the morphism $ G $ in (\ref{eq:supportsequence}), i.e. it's the scheme $ D_{7}(G) = \{ y \in \mathbf{P}_{2} \, | \, rank(G_{y}) \leq 7  \}. $
According to \cite{O0}, the expected codimension over $ \mathbf{P}_{2} $ of $ D_{7}(G) $ is $ 2 $, that is we expect a finite number of unstable lines for $ \Omega_{\mathbf{P}^2}^{1}(\log \mathcal{D}) $. Assuming that $ D_{7}(G) $ is $ 0 $-dimensional, the number of its points is given by Porteous' formula:
$$ [D_{7}(G)] = det [ 
c_{1-i+j}(((\Omega_{\mathbf{P}_{2}}^{1})^{3} \oplus \mathcal{O}_{\mathbf{P}_{2}}(-1)^{2}) - R^{1}\overline{q}_{\ast}\overline{p}^{\ast}(\mathcal{O}_{\mathbf{P}^2}(-4)^3))],  $$
where $ 1 \leq i,j \leq 2 $. The generic entry of $ [D_{7}(G)] $ is the coefficient of the degree-$ (1-i+j) $ term of the formal series in one variable defined as the quotient of the Chern polynomial of $ (\Omega_{\mathbf{P}_{2}}^{1})^{3} \oplus \mathcal{O}_{\mathbf{P}_{2}}(-1)^{2} $ with the one of $ R^{1}\overline{q}_{\ast}\overline{p}^{\ast}(\mathcal{O}_{\mathbf{P}^2}(-4)^3) $. So $ [D_{7}(G)] = 21 $.
These arguments imply the following:
\begin{prop}\label{p:21lines}
Let $ \mathcal{D} = \{C_{1}, C_{2}, C_{3}\} $ be a normal crossing arrangement of conics in $ \mathbf{P}^{2} $. $ UL(\Omega_{\mathbf{P}^2}^{1}(\log \mathcal{D})) $ is expected to be a 0-dimensional scheme of $ \mathbf{P}_{2} $ with 21 points.
\end{prop}
\begin{rem}
The previous proposition holds, more generally, for all vector bundles $ E $ over $ \mathbf{P}^{2} $ admitting the exact sequence
$$ 0 \longrightarrow \mathcal{O}_{\mathbf{P}^{2}}(-2)^{3} \longrightarrow \mathcal{O}_{\mathbf{P}^{2}}(-1)^{3} \oplus \mathcal{O}_{\mathbf{P}^{2}}^{2} \longrightarrow E \longrightarrow 0.  $$
\end{rem}
By using Macaulay2 software system, we can find some examples of arrangements that behave as stated in Proposition \ref{p:21lines}. \\
\noindent{\bf Example 2}\\
\noindent Let us consider the arrangement $ \mathcal{D}_{0} = \{C_{0,1},C_{0,2},C_{0,3}\} $ of conics with normal crossings such that
$$ C_{0,1}: x_{0}^{2}+x_{1}^{2}-16x_{2}^{2} = 0, $$
$$ C_{0,2}: x_{0}^{2}+9x_{1}^{2}-36x_{2}^{2} = 0, $$
$$ C_{0,3}: 25x_{0}^{2}+100x_{0}x_{2}+x_{1}^{2}-2x_{1}x_{2}+76x_{2}^{2} = 0. $$
In order to determine $ UL(\Omega_{\mathbf{P}^2}^{1}(\log \mathcal{D}_{0})) $, we contruct the $ 3 \times 5 $ matrix associated to $ \Omega_{\mathbf{P}^2}^{1}(\log \mathcal{D}_{0})^{\vee} $ in the sense of (\ref{eq:Anconas}) and then we restrict it to a generic line $ L \subset \mathbf{P}^{2} $ parametrized by $ x_{0} = bx_{1}+cx_{2} $. Afterwards we produce the $ 9 \times 8 $ matrix $ M $ of the map
$$ H^{0}(L, \mathcal{O}_{\mathbf{P}^{2}}(1)^{3}_{|_{L}} \oplus {\mathcal{O}_{\mathbf{P}^{2}}^{2}}_{|_{L}}) \longrightarrow H^{0}(L, \mathcal{O}_{\mathbf{P}^{2}}(2)^{3}_{|_{L}}) $$ with respect to the basis given by $ \{\{x_{1},x_{2}\}, \{x_{1},x_{2}\}, \{x_{1},x_{2}\}, \{1\}, \{1\}\} $ and $ \{\{x_{1}^2, x_{1}x_{2}, x_{2}^{2}\}, \{x_{1}^2, x_{1}x_{2}, x_{2}^{2}\}, \{x_{1}^2, x_{1}x_{2}, x_{2}^{2}\} \} $. Since we are interested in the kernel of this linear map, we consider, over the ring $ k[b,c] $, where for simplicity $ k  = \mathbf{Q} $, the ideal $ J $ generated by the $ 8 $ maximal minors of $ M $: as expected, we get that $ J $ has dimension $ 0 $ and degree $ 21 $, in particular the algebra $ k[b,c] / J $ is a $ \mathbf{C} $-vector space of dimension $ 21 $. Therefore, let $ B $ a basis of $ k[b,c] / J $ and let $ compb $ (resp. $ compc $) be the $ 21 \times 21 $ matrix associated, with respect to $ B $, to the linear map (\emph{companion})
$$  k[b,c] / J  \longrightarrow  k[b,c] / J  $$  
defined by the multiplication by $ b $ (resp. by $ c $). According to \cite{CLO} (chapter 2, section 4), the eigenvalues of \emph{compb} (resp. \emph{compc}) coincide with the $ b $-coordinates (resp. $ c $-coordinates) of the points of the variety associated to $ J $. By using Stickelberger's Theorem (see for example \cite{St}, chapter 2, section 2.3), in order to get the pair of parameters $ (b_{i},c_{i}) $ defining an unstable line it suffices to match the eigenvalue $ b_{i} $ of \emph{compb} corresponding to the same eigenvector (up to a change of sign) of the eigenvalue $ c_{i} $ of \emph{compc}. \\
If $ \mathcal{D}_{0} $ is as above, then $ \Omega_{\mathbf{P}^2}^{1}(\log \mathcal{D}_{0}) $ has $ 21 $ unstable lines such that $ 11 $ are real.
\begin{figure}[h]
    \centering
    \includegraphics[width=130mm]{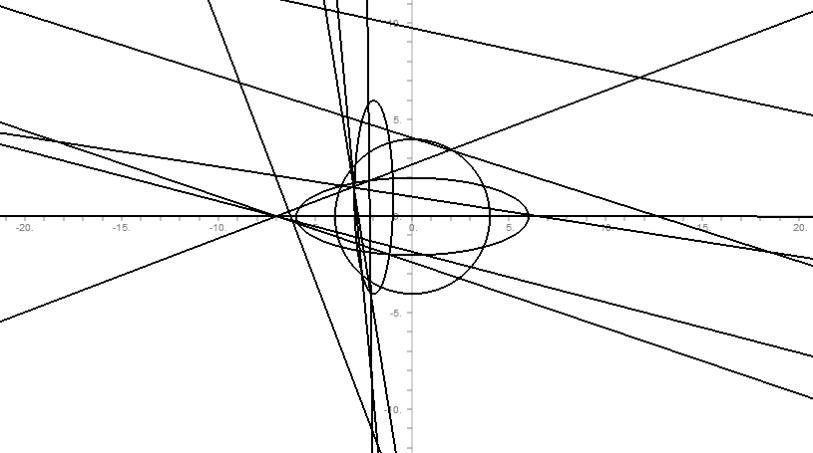}
    \caption{$ \mathcal{D}_{0} $ and the $ 11 $ real unstable lines plotted with \cite{SS}}
\end{figure}
\begin{rem} 
As we can see in Figure $ 3 $, it seems to be hard but interesting to understand what these lines represent for the conic-arrangement and how it is possible to get the conics from them: we observe, for example, that they are not tangent lines and they don't cross the conics in special points. So we can say that the three conics case represents still an open problem.    
\end{rem}

\end{document}